\documentclass[a4paper,twoside,draft]{article}

\author{Yves Lafont\thanks{Universit\'e de la M\'editerran\'ee (Aix-Marseille 2) - Institut de Math\'ematiques de Luminy (UMR 6206 du CNRS)},\  Fran\c cois M\'etayer\thanks{Universit\'e Denis Diderot (UMR 7126 du CNRS)  - Laboratoire Preuves, Programmes et Syst\`emes} \&\  Krzysztof Worytkiewicz\thanks{Akademia G\'orniczo-Hutnicza Krak\'ow - Katedra Informatyki}}

\usepackage[T1]{fontenc}
\usepackage{times,amssymb, amsmath, rotating}
\usepackage[all,2cell,frame]{xy}
\UseAllTwocells

\usepackage{enumerate} 

\newcommand \ms \medskip
\newcommand \sssec {\subsubsection*{}}

\usepackage[amsmath,thmmarks]{ntheorem}

\theoremstyle{plain} 
\theoremheaderfont{\normalfont\bfseries}
\theorembodyfont{\slshape} 
\theoremseparator{.}
\theoremsymbol{} 
\newtheorem{theorem}{Theorem} 
\newtheorem{proposition}{Proposition} 
\newtheorem{lemma}{Lemma}
\newtheorem{corollary}{Corollary}

\theoremstyle{plain} 
\theoremheaderfont{\normalfont\bfseries}
\theorembodyfont{\slshape} 
\theoremseparator{.}
\theoremsymbol{}
\newtheorem{defn}{Definition}

\theoremstyle{plain} 
\theoremheaderfont{\slshape}
\theorembodyfont{\upshape} 
\theoremseparator{.}
\theoremsymbol{\ensuremath{\lozenge}}
\newtheorem{rem}{Remark}

\theoremheaderfont{\slshape}
\theorembodyfont{\upshape} 
\theoremstyle{nonumberplain} 
\theoremseparator{.} 
\theoremsymbol{\ensuremath{\triangleleft}} 
\newtheorem{proof}{Proof} 

\theoremheaderfont{\slshape}
\theorembodyfont{\upshape} 
\theoremstyle{nonumberplain} 
\theoremseparator{\hspace{-.5ex}} 
\theoremsymbol{\ensuremath{\triangleleft}} 
\newtheorem{proofthm}{Proof of Theorem}

\theoremstyle{empty} 

\theoremsymbol{} 
\theoremseparator{} 
\theorembodyfont{\slshape}
\theoremprework{\begin{quote}} 
\theorempostwork{\end{quote}\medskip} 
\newtheorem{stmt}{}

\newcommand{\dueto}[1]{\textup{\textbf{(#1) }}}

\newcommand{\tmem}[1]{{\em #1\/}}

\newcommand{\tmop}[1]{\ensuremath{\operatorname{#1}}}

\newcommand{\tmtextbf}[1]{{\bfseries{#1}}}
\newcommand{\tmtextit}[1]{{\itshape{#1}}}
\newcommand{\upl}{+}

\newenvironment{enumerateroman}{\begin{enumerate}[i.] }{\end{enumerate}}

\newenvironment{itemizeminus}
{\begin{itemize}

}
{\end{itemize}}

\newcommand \NN {\mathbb N} 
\newcommand \id {\mathrm{id}} 

\newcommand \setbis[2] {\{#1\; |\; #2\}} 

\newcommand \ctg[1] {\hbox{\bf #1}} 
\newcommand \morph[1] {#1^{\arrowcat}} 
\newcommand \arrowcat {(\cdot\to\cdot)} 
\newcommand \cat {\ctg{Cat}} 
\newcommand \glob {\ctg{O}} 
\newcommand \globset  {\ctg{Glob}}  
\newcommand \nglobset[1]{#1\globset}  
\newcommand \sets {\ctg{Sets}} 
\newcommand \ncat[1] {#1{\cat}} 
\newcommand \ncatpl[1] {#1{\cat}^{+}} 
\newcommand \ocat {\ncat{\omega}} 
\newcommand \twocat {\ncat{2}} 
\newcommand \pol {\ctg{Pol}} 

\newcommand \opp[1] {{#1}^{\mathrm{op}}} 

\newcommand \II {F} 
\newcommand \SH {G} 
\newcommand \TT {U} 
\newcommand \TTT {V} 
\newcommand \GG {T} 

\newcommand \HH[2] {\mathrm H_{#1}(#2)} 
\newcommand \free[1] {#1^{\ast}} 

\newcommand \I {I} 
\newcommand \J {J}

\newcommand \Tfb {\I{-}\mathrm{inj}} 
\newcommand \Cof {\I{-}\mathrm{cof}} 
\newcommand \Weq {\mathcal W} 
\newcommand \Imm {\mathcal Z} 
\newcommand \Icof[1]{#1{-}\mathrm{cof}} 
\newcommand \inj[1]{#1{-}\mathrm{inj}} 
\newcommand \Ifib[1]{#1{-}\mathrm{inj}} 
\newcommand \Icell[1]{#1{-}\mathrm{cell}} 

\newcommand \llp[1] {^{\pitchfork} \! #1} 
\newcommand \rlp[1] { #1 \,^{\pitchfork}} 

\newcommand \Morph[1] {{#1}^{\rightarrow}} 
\newcommand \dom {\tmop{dom}}  
\newcommand \cod {\tmop{cod}}  

\newcommand \para \parallel
\newcommand \To[1] {\to_{#1}}
\newcommand \Comp[1] {\ast_{#1}}
\newcommand \unit[1] {1_{#1}}
\newcommand \Unit[2] {1^{#1}_{#2}}

\newcommand \cons {\mathrel{\triangleright}} 
\newcommand \Cons[1] {\cons_{#1}} 
\newcommand \Sce[1] {#1^\flat} 
\newcommand \Tge[1] {#1^\sharp} 
\newcommand \sce[1] {\sigma_{#1}} 
\newcommand \tge[1] {\tau_{#1}} 
\newcommand \SCE[2] {\sce{#1,#2}} 
\newcommand \TGE[2] {\tge{#1,#2}} 
\newcommand \cosce[1] {{\mathrm s}_{#1}} 
\newcommand \cotge[1] {{\mathrm t}_{#1}} 

\newcommand \Hom[3] {#1(#2,#3)} 
\newcommand \HOM[2] {[#1,#2]} 
\newcommand \act \cdot 
\newcommand \COMP \circledast 
\newcommand \Sht[1] {{\left [#1 \right ]}} 

\newcommand \eqv \sim 
\newcommand \rto {\stackrel{\sim}\to}
\newcommand \inv \overline

\newcommand \cto \curvearrowright 
\newcommand \Cto[3] {#1 \; | \; #2 \cto #3}
\newcommand \Pal[1] {#1^\natural} 
\newcommand \comp \ast 
\newcommand \Top {\pi^1}
\newcommand \Bot {\pi^2} 
\newcommand \Triv \tau 

\newcommand \cnx \Gamma 
\newcommand \Cnx[1] {\cnx(#1)}

\newcommand \Glu[1] {\Pi(#1)}
\newcommand \lft \hat 
\newcommand \rht \tilde

\newcommand \SNG {\boldsymbol 1}
\newcommand \ZERO {\boldsymbol 0}
\newcommand \OO[1] {\mathbf O^{#1}} 
\newcommand \DO {\partial \OO} 
\newcommand \PP[1] {\mathbf P^{#1}} 
\newcommand \sng[1] {\langle #1 \rangle} 
\newcommand \pair[2] {\langle #1, #2 \rangle}
\newcommand \ii[1] {\mathbf i_{#1}}
\newcommand \jj[1] {\mathbf j_{#1}}
\newcommand \jjj[1] {\mathbf j'_{#1}}
\newcommand \kk[1] {\mathbf k_{#1}}
\newcommand \oo[1] {\mathbf o_{#1}}
\newcommand \pp[1] {\mathbf p_{#1}}

\newcommand \uar {\ar@/^1pc/}
\newcommand \dar {\ar@/_1pc/}
\newcommand \Uar {\ar@/^2pc/}
\newcommand \Dar {\ar@/_2pc/}
\newcommand \ear {\ar@{-->}}
\newcommand \edar {\ar@{-->}@/_1pc/}

\newcommand \OT \leftleftarrows 

\newcommand{\doubl}[2]{\ar@<2pt>[l]^{#2}\ar@<-2pt>[l]_{#1}}

\newcommand{\doubr}[2]{\ar@<2pt>[r]^{#1}\ar@<-2pt>[r]_{#2}}
\newcommand{\doubld}[2]{\ar@<2pt>[ld]^{#2}\ar@<-2pt>[ld]_{#1}}
\newcommand{\doubd}[2]{\ar@<2pt>[d]^{#2}\ar@<-2pt>[d]_{#1}}
\newcommand{\Doubld}{\doubld{}{}}

\newcommand \condi[1]{\textbf{#1}\ }
\newcommand \mone{\condi{(M1)}}
\newcommand \mtwo{\condi{(M2)}}
\newcommand \mthree{\condi{(M3)}}
\newcommand \mfour{\condi{(M4)}}
\newcommand \cone{\condi{(C1)}}
\newcommand \ctwo{\condi{(C2)}}
\newcommand \cthree{\condi{(C3)}}
\newcommand \cfour{\condi{(C4)}}
\newcommand \cfive{\condi{(C5)}}
\newcommand \sone{\condi{(S1)}}
\newcommand \stwo{\condi{(S2)}}
\newcommand \sthree{\condi{(S3)}}
\newcommand \sfour{\condi{(S4)}}
\newcommand \zone{\condi{(Z1)}}
\newcommand \ztwo{\condi{(Z2)}}
\newcommand \zthree{\condi{(Z3)}}
\newcommand \zthreep{\condi{(Z3')}}

\newcommand \abso[1]{\left |{#1}\right |}

\newcommand \weq{$\omega$-weak equivalence}
\newcommand \weqs{$\omega$-weak equivalences}

\newcommand \deq{\overset{\tmop{def} .}{=}}

\title{A folk model structure on omega-cat}

\begin{document}

\maketitle

\begin{abstract}

 The primary aim of this work is an intrinsic homotopy
 theory of strict $\omega$-categories. We
 establish a model structure on $\ocat$, the 
category of strict $\omega$-categories. The constructions leading to
the model structure in question are
expressed entirely within the scope of $\ocat$, building on a
set of generating cofibrations and a class of weak equivalences as
basic items. All object are fibrant while free objects
are cofibrant. We further exhibit model structures of this type on
$n$-categories for arbitrary $n \in \NN$, as specialisations of the
$\omega$-categorical one along right adjoints. In
particular, known cases for $n=1$ and $n=2$ nicely fit into the scheme.
\end{abstract}

\section{Introduction}\label{sec:intro}

\subsection{Background and motivations}\label{subsec:backgd}
The origin of the present work goes back to the following result~\cite{anick:homasa,squier:worphf}:
\begin{stmt}
  if a monoid $M$ can be presented by a finite, confluent and terminating
  rewriting system, then its third homology group $\HH 3M$ is of finite type.
\end{stmt}
The finiteness property extends in fact to all dimensions~\cite{kobayashi:comrsh}, but the above theorem may also be refined in another direction: the same hypothesis implies that $M$ has {\em finite derivation type}~\cite{squieral:fincrs}, a property of homotopical nature.

We claim that these ideas are better expressed in terms of $\omega$-categories (see~\cite{guiraud:trdimp,guiraud:twoppp,lafont:alggrw}). Thus we work in the category $\ocat$, whose objects are the strict $\omega$-categories and the morphisms are $\omega$-functors (see Section~\ref{sec:omegacat}). In fact, when considering the interplay between the monoid itself and the space of computations attached to any presentation of it, one readily observes that both objects support a structure of $\omega$-category in a very direct way: this was the starting point of~\cite{metayer:respol}, which introduces a notion of {\em resolution} for $\omega$-categories, based on computads~\cite{street:limicf,power:ncatpt} or polygraphs~\cite{burroni:highdw}, the terminology we adopt here. Recall that a polygraph $S$ consists of sets of cells of all dimensions, determining a freely generated $\omega$-category $\free S$. A resolution of an $\omega$-category $C$ by a polygraph $S$ is then an $\omega$-functor $p:\free S\to C$ satisfying a certain lifting property (see Section~\ref{sec:cofibrant} below); \cite{metayer:respol} also defines a homotopy relation between $\omega$-functors and shows that any two resolutions of the same $\omega$-category are homotopically equivalent in this sense.

This immediately suggests looking for a homotopy theory on $\ocat$ in
which the above resolutions become trivial fibrations: the model
structure we describe here does exactly that. Notice, in addition,
that polygraphs turn out to be the cofibrant objects
(see~\cite{metayer:cofohc} and Section~\ref{sec:cofibrant} below). On
the other hand, our model structure generalizes in a very precise
sense the ``folk'' model structure on $\cat$
(see~\cite{joyaltierney:strscs}) as well a model structure on $\ncat
2$ in a similar spirit (see~\cite{lack:quitwo,lack:quibic}). Incidentally, there is also a
quite different, Thomason-like, model structure on $\ncat
2$ (see~\cite{worytkiewiczetal:modstc}). Its generalisation to $\ocat$
remains an open problem.

Since~\cite{quillen:homalg}, the notion of model structure has been gradually recognized as the appropriate abstract framework for developing homotopy theory in a category $\ctg{C}$: it consists in three classes of morphisms, {\em weak equivalences}, {\em fibrations}, and {\em cofibrations}, subject to axioms whose exact formulation has somewhat evolved in time. In practice, most model structures are {\em cofibrantly generated}. This means that there are sets $\I$ of {\em generating cofibrations} and $\J$ of {\em generating trivial cofibrations} which determine all the cofibrations and all the fibrations by lifting properties. 

Recall that, given a set $\I$ of morphisms, {\em  $\I$-injectives} are the morphisms which have the {\em right lifting property} with respect to $\I$. They build a class denoted by $\Tfb$.  Likewise, {\em $\I$-cofibrations} are the morphisms having the {\em left lifting property } with respect to $\Tfb$ (see Section~\ref{subsec:smallobject}). The class of $\I$-cofibrations is denoted by $\Cof$.  Now, our construction is based on a theorem by J.Smith (see~\cite{beke:shhomc}):
under some fairly standard assumptions on the underlying category, conditions
\begin{stmt}
  \begin{description}
    \item[(S1)] $\Weq$ has the $3$ for $2$ property and is stable under
    retracts;
    \item[(S2)] $\Ifib{I} \subseteq \Weq$;
    \item[(S3)] $\Icof{I} \cap \Weq$ is closed under
    pushouts and transfinite compositions;
    \item[(S4)] $\Weq$ admits a solution set $J \subseteq \Icof{I}\cap \Weq$ at $I$.
  \end{description}
\end{stmt}
are sufficient to obtain a model structure in which $\Weq$, $\I$ and
$\J$ are  the  weak equivalences, the generating cofibrations and the
generating trivial cofibrations, respectively.

\subsection{Organization of the paper}\label{subsec:organize}
Section~\ref{sec:combi}  reviews {\em combinatorial model categories}, with special emphasis on our version of Smith's theorem (Section~\ref{subsec:solset}), while Section~\ref{sec:omegacat} recalls the basic definitions of globular sets and $\omega$-categories, and sets the notations. 

Section~\ref{sec:folkmodel} is the core of the paper, that is the
derivation of our model structure by means of a set $\I$ of generating
cofibrations and a class $\Weq$ of weak equivalences, satisfying conditions~\sone to \sfour.

\subsubsection{Sketch of the main argument}\label{subsubsec:sketch}
We first define the set $\I$ of generating cofibrations, and establish
closure properties we shall use later in the proof of
condition~\sthree.

We then define the class $\Weq$ of \weqs, which are at this stage our candidates for the r\^ole
of weak equivalences~(Section~\ref{subsec:weq}). For this purpose, we first need a
notion of $\omega$-equivalence between parallel
cells~(Section~\ref{subsec:equivalence}), together with crucial
properties of this notion.

We then prove condition~\stwo, and part of~\sone (Section~\ref{subsec:weq}), as well as additional closure properties contributing to~\sthree.

At this stage, just one point of~\sone remains unproved, namely the assertion
\begin{stmt}
  if $f:X\to Y$ and $g\circ f:X\to Z$ belong to $\Weq$, then so does $g:Y\to Z$.
\end{stmt}

This requires an entirely new construction: we define an endofunctor $\cnx$ of $\ocat$, which to each $\omega$-category $X$ associates an $\omega$-category $\Cnx X$ of {\em reversible cylinders} in $X$. Section~\ref{subsec:connect} summarizes the main features of $\cnx$, whereas the more technical proofs are given in Appendix~\ref{annex:connect}. This eventually leads to an alternative characterization of weak equivalences and to a complete proof of~\sone.

As for condition~\sthree, the difficult point is to prove the closure of $\Cof\cap\Weq$ by pushout, which does not follow from the previously established properties. The main obstacle is that $\Weq$ itself is definitely {\em not} closed by pushout. What we need instead is a new class $\Imm$ of {\em immersions} such that:
\begin{enumerateroman}
\item $\Imm$ is closed by pushout;
\item $\Cof\cap\Weq \subseteq \Imm\subseteq \Weq$,
\end{enumerateroman}
which completes the proof of~\sthree. Immersions are defined in Section~\ref{subsec:immersions}, by using again the functor $\cnx$ in an essential way.

Section~\ref{subsec:generic} is devoted to the proof of the {\em solution set condition}~\sfour. Precisely, we have to build, for each $i\in\I$, a {\em set} $\J_i$ of $\omega$-functors satisfying the following property: for each commutative square
\begin{equation}
  \begin{xy}
    \xymatrix{X\ar[r]\ar[d]_i & Z\ar[d]^f \\
              Y\ar[r] & T}
  \end{xy}
\label{eq:solset}
\end{equation}
 where $i\in\I$ and $f\in\Weq$, there is a $j\in\J_i$ such that~(\ref{eq:solset}) factors through $j$:
\begin{equation}
  \begin{xy}
    \xymatrix{X\ar[r]\ar[d]_i & U\ar[d]_j\ar[r] & Z\ar[d]^f \\
              Y\ar[r] &V\ar[r] & T}
  \end{xy}
\label{eq:solsetfac}
\end{equation}
The whole solution set is then $J \; \deq \; \bigcup_{i\in \I}\J_i$. It turns out that in our case, the sets $\J_i$ are just singletons.

\subsubsection{Additional properties}\label{subsubsec:additional}
The end of the paper is devoted to two additional points: Section~\ref{sec:cofibrant} gives a characterization of cofibrant objects as polygraphs by interpreting the results of~\cite{metayer:cofohc} in terms of our model structure. Finally, Section~\ref{sec:ncat} shows how the present model structure on $\ocat$ transfers to $\ncat n$ for any integer $n$: in particular, for $n=1$ and $n=2$, we recover the abovementioned structures on $\cat$~\cite{joyaltierney:strscs} and $\ncat 2$~\cite{lack:quitwo,lack:quibic}.

\section{Combinatorial model categories}\label{sec:combi}

We recall some facts about model categories with locally-presentable
underlying categories.

\subsection{Locally presentable categories}\label{subsec:locpres}

Let $\alpha$ be a regular cardinal. An $\alpha$-filtered category
$\ctg{F}$ is a category such that

\begin{enumerateroman}

  \item for any set of objects $S$ with cardinality $|S| < \alpha$
 and for each $A
  \in S$ there is an object $T$ and a morphism $f_A: A \rightarrow T$;
  
  \item for any
    set of morphisms $M \subseteq \text{\tmtextbf{F}}(A,B)$ with cardinality
  $|M| < \alpha$ there is an object $C$ and a morphism $m : B
  \rightarrow C$, 
  such that $m \circ m' = m \circ m''$ for all $m',m'' \in M$.
\end{enumerateroman}

We say that $\ctg{F}$ is filtered in case $\alpha = \aleph_0$. In
particular, a {\tmem{directed}} (partially ordered) set is a filtered
category.

Recall that an $\alpha$-filtered colimit is a colimit of a functor $D :
\text{\tmtextbf{I}} \rightarrow \ctg{C}$ from a small
$\alpha$-filtered category $\ctg{I}$. Let $\ctg{C}$ be
a category. An object $X \in \ctg{C}$ is $\alpha$-presentable if
the covariant representable functor $\ctg{C} (X, -) :
\ctg{C} \rightarrow \sets$ preserves
$\alpha$-filtered colimits. This boils down to the fact that a morphism from
$X$ to an $\alpha$-filtered colimit factors through some object of the
relevant $\alpha$-filtered diagram, in an essentially unique way. If $X$ is $\alpha$-presentable, and $\beta$ is a regular cardinal such that $\alpha<\beta$, then $X$ is also $\beta$-presentable.

We say that an object $X \in
\ctg{C}$ is {\tmem{presentable}} if there is a regular cardinal
witnessing this fact. If it is the case, the smallest such cardinal, $\pi
(X)$, is called $X$'s {\tmem{presentation rank}}.

\begin{defn}
  \label{def:loc-pres}Let $\alpha$ be a regular cardinal. A cocomplete
  category $\ctg{C}$ is {\em locally $\alpha$-presentable} if there is a
  family $G = (G_i)_{i \in I}$ of objects such that every object of
  $\ctg{C}$ is an $\alpha$-filtered colimit of a diagram in the
  full subcategory spanned by the $G_i$'s. We say that a cocomplete
  category is {\em locally finitely presentable} if it is locally
  $\aleph_0$-presentable. Finally, we say that a cocomplete
  category is {\em locally presentable} if there is a regular
  cardinal witnessing this fact.
\end{defn}

Definition \ref{def:loc-pres} is equivalent to the original one by Gabriel and
Ulmer {\cite{gabrielulmer:lokprc}}. It proves especially powerful to establish factorisation results, when combined with the {\em small object argument} (Section~\ref{subsec:smallobject}).
Let $\beta$ be a regular cardinal. Recall that a $\beta$-colimit is a colimit of a functor $D : \ctg{I} \rightarrow \ctg{C}$ from a small category $\ctg{I}$ such that $\abso{\ctg{I}_1}< \beta$.
\begin{proposition}
  \label{prop:borceux}Let $\beta$ be a regular cardinal. A $\beta$-colimit of
  $\beta$-presentable objects is $\beta$-presentable.
\end{proposition}
\begin{rem}
  \label{rem:loc-pres}Let $\alpha$ be a regular cardinal and
  $\ctg{C}$ be a locally $\alpha$-presentable category. By
  definition of local presentability, every object $X \in \ctg{C}$
  is an $\alpha$-filtered colimit of a diagram of $\alpha$-presentable
  objects, so it is a $\beta$-colimit for a regular cardinal $\beta$ such that
  $\alpha \leqslant \beta \leqslant {\abso{\ctg{C}_1}}^+$ . Thus, by
  virtue of Proposition~\ref{prop:borceux}, every object of $\ctg{C}$ is presentable (with a presentation rank possibly exceeding $\alpha$).
\end{rem}

\subsection{Small objects for free}\label{subsec:smallobject}

Let $\ctg{C}$ be a category. Recall that its {\tmem{category of
morphisms}} $\Morph{\ctg{C}}$ is defined as the functor
category $\morph{\ctg{C}}$, where $\arrowcat$
is the category generated by the one-arrow graph. Let $f : X \rightarrow Y$
and $g : Z \rightarrow T$ be morphisms in $\ctg{C}$. We say that
$f$ has the {\tmem{left-lifting}} property with respect to
$g$, or equivalently that $g$ has the {\tmem{right lifting}} property with respect to $f$, if every commuting square $(u, v) \in \Hom{\Morph{\ctg{C}}}{f}{g}$ admits a {\tmem{lift}},  that is a morphism $h : Y \rightarrow T$ making the following diagram commutative
\begin{center}
  $\xymatrix{
X \ar[r]^u \ar[d]_f & 
Z \ar[d]^g
\\
Y \ar@{.>}[ur]^h \ar[r]_v & T
}$
\end{center}
This relation is denoted by $f \pitchfork g$.

For any class of morphisms $\mathcal{A}$, we define
\begin{eqnarray*}
  \llp{\mathcal{A}} & \deq  & \setbis{f}{ f\pitchfork g, g \in \mathcal{A}}\\
  \rlp{\mathcal{A}} & \deq & \setbis{g}{f\pitchfork g, f \in \mathcal{A}}
\end{eqnarray*}
\begin{proposition}
  \label{prop:retracts}Suppose $f = f'' \circ f'$. Then
  \begin{itemizeminus}
    \item if $f' \pitchfork f$ then $f$ is a retract of $f''$;
    
    \item if $f \pitchfork f''$ then $f$ is a retract of $f'$.
  \end{itemizeminus}
\end{proposition}

Proposition~\ref{prop:retracts} is known as ``the retract argument''.

Let $\dom : \Morph{\ctg{C}} \rightarrow \ctg{C}$
and $\cod : \Morph{\ctg{C}} \rightarrow
\ctg{C}$ be the obvious functors picking the domain and the
codomain of a morphism, respectively. A {\tmem{functorial factorisation}} in
$\ctg{C} $ is a triple
\[ F = (F, \lambda, \rho) \]
where $F : \Morph{\ctg{C}} \rightarrow \ctg{C}$
is a functor while $\lambda : \tmop{dom} \rightarrow F$ and $\rho : F
\rightarrow \tmop{cod}$ are natural transformations. Let $\mathcal{L}$ and
$\mathcal{R}$ be classes of morphisms in $\ctg{C} .$ We say that
the pair $(\mathcal{L}, \mathcal{R})$ admits a functorial factorisation $(F,
\lambda, \rho)$ provided \ that $\lambda_f \in \mathcal{L}$ and $\rho_f \in
\mathcal{R}$ for all morphisms $f \in \Morph{\ctg{C}}$. If
$(F, \lambda, \rho)$ is clear from the context (or if it does not matter), we
say by abuse of language that $(\mathcal{L}, \mathcal{R})$ {\tmem{is}} a
functorial factorisation.

Let $I$ be a set of morphisms in a cocomplete category $\ctg{C}$
and $I^{\ast}$ be the closure of $I$ under pushout. The class
$I$-cell of {\tmem{relative}} \tmtextit{$I$-cell complexes} is the
closure of $\text{$I^{\ast}$}$ under \ transfinite composition. Let 
$\inj{\I} \deq \rlp{I}$ and $\Icof{I} \deq \llp{(\inj{\I})}$.

\begin{rem} 
  \label{rem:inj-cell-cof} If $I \subseteq I'$, then $\Ifib{I}\supseteq \Ifib{I'}$ and $\Ifib{J} = \Ifib{(\Icof{J})}$. It is easy to see that $\Icell{I}\subseteq \Icof{I}$.
\end{rem}

The next proposition recalls standard formal properties of the classes just defined (see~\cite{gabrielzisman:calfrh}).

\begin{proposition}\label{prop:closure}
$\Tfb$ as well as $\Cof$ contain all identities. $\Tfb$ is closed under composition and pullback while $\Cof$ is closed under retract, transfinite composition and pushout.
\end{proposition}

We may now state the crucial factorisation result we shall need:

 \begin{proposition}\label{prop:smallobject} Suppose that $\ctg{C}$ is locally presentable and let $I$ be a set of morphisms of $\ctg{C}$. Then $(\Icell{I},\Ifib{I})$ is a functorial factorisation.
\end{proposition}

\begin{proof}
  The required factorisation is produced by the ``small object argument'', due to Quillen (see also~\cite{garner:undsoa} for an extensive discussion): 
  \begin{itemizeminus}
    \item For any $f$ in $\Morph{\ctg{C}}$, let $S_f$ be the set of morphisms of          $\Morph{\ctg{C}}$ with domain in $I$ and codomain $f$, that is
      \begin{displaymath}
        S_f=\setbis{s=(u_s,v_s)\in\Morph{\ctg{C}}}
                   {\dom (s)=i_s\in I,\ \cod (s)=f}.
      \end{displaymath}
We get a functor $F:\Morph{\ctg{C}}\to\ctg{C}$ together with natural transformations $\lambda:\dom \to F$ and $\rho:F\to\cod$ determined by the inscribed pushout of the outer commutative square
\begin{center}
 $\xygraph{ 
!{<0cm,0cm>;<1cm,0cm>:<0cm,1cm>::}
!{(0,0) }*+{{\displaystyle \coprod_{s\in S_f} A_s}}="1"
!{(0,-2.8) }*+{{\displaystyle \coprod_{s \in S_f} B_s}}="2"
!{(3.6,0) }*+{X}="3"
!{(3.6,-2.8)}*+{Y}="4"
!{(2.4,-2)}*+{F (f)}="5"
"1":"2"_{{\displaystyle \coprod_{s \in S_f} i_s}}
"1":"3"^{[(u_s)_{s\in S_f}]}
"2":"4"_{[(v_s)_{s\in S_f}]}
"3":"4"^{f}
"2":"5"^{j_0}
"3":"5"_{\lambda_f}
"5":@{.>}"4"^{\rho_f}
!{(1.8,-1.45)}*+{}="pusha" 
!{(2.05,-1.45)}*+{}="pushb" 
!{(1.8,-1.7)}*+{}="pushc" 
"pusha":@{-}"pushb" 
"pusha":@{-}"pushc"
}$
\end{center}
where, for each $s\in S_f$, $i_s:A_s\to B_s$, and $[(u_s)_{s\in S_f}]$, $[(v_s)_{s\in S_f}]$ are given by the universal property of coproducts.

\item By transfinite iteration of the previous construction, we get, for each ordinal $\beta$, a triple $(F^{\beta},\lambda^{\beta},\rho^{\beta})$. Precisely,
  \begin{eqnarray*}
    F^{0}(f) & \deq & F(f),\\
    \lambda^0_f & \deq & \lambda_f, \\
    \rho^0_f    & \deq & \rho_f;
  \end{eqnarray*}
if $\beta+1$ is a successor ordinal, then
  \begin{eqnarray*}
    F^{\beta + 1} (f) & \deq & F \left( \rho^{\beta}_f
    \right), \\
    \lambda^{\beta + 1}_f & \deq  &
    \lambda_{\rho^{\beta}_f} \circ \lambda^{\beta}_f, \\
    \rho^{\beta \upl 1}_f & \deq & \rho_{\rho^{\beta}_f},
  \end{eqnarray*}
  and if $\beta$ be a limit ordinal, then
  \[ \text{$F^{\beta} (f) \deq \tmop{colim}_{\gamma <
     \beta} F^{\gamma} (f)$} \]
  while $\lambda^{\beta}_f$ and $\rho^{\beta}_f$ are given by transfinite
  composition and universal property, respectively.
\item Now notice that, for each ordinal $\beta$, $\lambda^{\beta}_{f}$ belongs to
$\Icell{I}$, and that $(\lambda^{\beta}_f,\rho^{\beta}_f)$ is a functorial factorisation. It remains to show that there is an ordinal $\kappa$ for which $\rho^{\kappa}_f$ belongs to $\Ifib{I}$. This is where local presentability helps: thus, let $\kappa$ be a regular cardinal such that for each $i\in I$, the presentation rank $\pi (\dom i)$ is strictly smaller than $\kappa$, and suppose that the outer square of the following diagram commutes:
\begin{center}
$\xygraph{ 
!{<0cm,0cm>;<1cm,0cm>:<0cm,1cm>::}
!{(0,0) }*+{A}="A"
!{(0,-4) }*+{B}="B"
!{(4,0) }*+{F^\kappa (f)}="Fk"
!{(4,-4)}*+{Y}="Y"
!{(1.3,-2.7) }*+{{\displaystyle \coprod_{s \in S_{\rho^\beta_f}} B_s}}="coprod"
!{(2.6,-1.4) }*+{F^{\beta + 1} (f)}="Fb"
"A":"B"_{i}
"A":"Fk"^{u}
"B":"Y"_{v}
"Fk":"Y"^{\rho^\kappa_f}
"B":@{.>}"coprod"_(.4){in_B}
"coprod":@{.>}"Fb"^(.5){j_{\beta + 1}}
"Fb":@{.>}"Fk"^{c^{\beta + 1,\kappa}}
}$
\end{center}
Since $A$ is $\kappa$-presentable and $F^{\kappa} (f)$ is a
$\kappa$-filtered colimit, there is a $\beta < \kappa$ such that $u$ factors
through $F^{\beta} (f)$ as $u = c_{\beta, \kappa} \circ u'$ for some $u'$,
with $c_{\beta, \kappa} : F^{\beta} (f) \rightarrow F^{\kappa} (f)$ the
colimiting morphism. It follows then from the above construction that
$c_{\beta + 1, \kappa} \circ j_{\beta + 1} \circ \tmop{in}_B $ is a lift, whence $\rho^{\kappa}_f\in\Ifib{I}$, and we are done. 
\end{itemizeminus}
\end{proof}

\subsection{Model structures and cofibrant generation}\label{subsec:cofgen}

We say that a class $\mathcal{A}$ of morphisms has the \tmtextit{$3$ for $2$
property} if whenever $h = g \circ f$ and any two out of the three morphisms
$f$, $g$, $h$ belong to $\mathcal{A}$, then so does the third. We now recall the basics of model structures, following the presentation of~\cite{hovey:modcat}.

\begin{defn}
  \label{def:modelstruc}A \tmtextit{model structure} on a complete and
  cocomplete category $\ctg{C}$ is given by three classes of
  morphisms, the class $\mathcal{C}$ of \tmtextit{cofibrations}, the class
  $\mathcal{F}$ of \tmtextit{fibrations}, and the class $\Weq$ of
  \tmtextit{weak equivalences}, satisfying the following conditions:
  \begin{description}
    \item[(M1)] $\Weq$ has the $3$ for $2$ property;
    
    \item[(M2)] $\mathcal{C}$, $\mathcal{F}$ and $\Weq$ are stable
    under retracts;
    
    \item[(M3)] $\mathcal{C} \cap \Weq \subseteq \; \llp{\mathcal{F}}
    $ and $\mathcal{F} \cap \Weq \subseteq \; \rlp{\mathcal{C}}$;
    
    \item[(M4)] the pairs $( \mathcal{C} \cap \Weq, \mathcal{F})$ and
    $( \mathcal{C}, \mathcal{F} \cap \Weq)$ are functorial
    factorisations.
  \end{description}
  A complete and cocomplete category equipped with a model structure is called
  a \tmtextit{model category}. The members of $\mathcal{F} \cap \Weq$
  are called \tmtextit{trivial fibrations} and the members of $\mathcal{C}
  \cap \Weq$ are \tmtextit{trivial cofibrations}.
\end{defn}

\begin{rem}
  There is a certain amount of redundancy in the definition of a model
  category as the class of fibrations is determined by the class of
  cofibrations and vice-versa: we have
  \begin{itemizeminus}
    \item $\mathcal{F}= \Ifib{( \mathcal{C} \cap \Weq)}$
    
    \item $\mathcal{F} \cap \Weq = \Ifib{\mathcal{C}}$;
  \end{itemizeminus}
  as well as
  \begin{itemizeminus}
    \item $\mathcal{C}= \; \llp{( \mathcal{F} \cap \Weq)}$;
    
    \item $\mathcal{C} \cap \Weq = \; \llp{\mathcal{F}}$.
  \end{itemizeminus}
\end{rem}

In most known model categories cofibrations and fibrations are generated by {\em sets} of morphisms. In the case of locally-presentable categories, we get the following definition:

\begin{defn}
  A locally-presentable model category is \tmtextit{cofibrantly generated} if
  there are two sets $I$, $J$ of morphisms such that
  \begin{enumerateroman}
    \item $\mathcal{C}= \Icof{I}$;
    
    \item $\mathcal{C} \cap \Weq = \Icof{J}$.
  \end{enumerateroman}
  The morphisms in $I$ are called {\tmem{\tmtextit{{\tmem{generating
  cofibrations}}}{\tmem{}}}} while the morphisms in $J$ are called \
  \tmtextit{{\tmem{generating trivial cofibrations}}}. Locally-presentable,
  cofibrantly generated model categories are called {\tmem{combinatorial model
  categories}}. 
\end{defn}

Notice that a locally-presentable model category is combinatorial if and only
if $\mathcal{F} \cap \Weq = \Ifib{I}$ and $\mathcal{F}=\Ifib{J}$. The whole point in the definition of combinatorial model categories is the possibility to apply the small object argument to arbitrary sets $I$ and $J$. The general case, however, requires extra conditions on those sets.

\subsection{The solution set condition}\label{subsec:solset}

Let $\ctg{C}$ be a category, $i$ a morphism of
$\ctg{C}$ and $\Weq$ a class of morphisms of
$\ctg{C}$. We say that $\Weq$ admits a \tmtextit{solution
set at $i$} if there is a \tmtextit{set} $W_i$ of morphisms such
that any commutative square

\begin{center}
  $\xymatrix{
\bullet \ar[d]_i \ar[r] & \bullet \ar[d]^{w \in {\mathcal W}} \\ 
\bullet \ar[r] & \bullet
}$

\end{center}
where $w \in \Weq$ factors through some $w' \in \Weq_i$:
\begin{center}
  $\xymatrix{
\bullet \ar[d]_i \ar[r] & 
\bullet\ar[r] \ar[d]|{w' \in W_i} &
\bullet \ar[d]^{w \in {\mathcal W}} 
\\ 
\bullet \ar[r] & \bullet \ar[r] & \bullet
}$
\end{center}
If $I$ is a set of morphisms, we say that $\Weq$ admits a \tmtextit{solution set at I} if it admits a solution set at any $i \in I$.  

We now turn to Smith's theorem, on which our construction is based:

\begin{theorem}
  \label{thm:smith} Let $I$ be a set, and $\Weq$ a class of morphisms in
  a locally presentable category $\ctg{C}$. Suppose that
  \begin{description}
    \item[(S1)] $\Weq$ has the $3$ for $2$ property and is stable under
    retracts;
    
    \item[(S2)] $\Ifib{I}\subseteq \Weq$;
    
    \item[(S3)] $\Icof{I} \cap \Weq$ is closed under
    pushouts and transfinite compositions;
    
    \item[(S4)] $\Weq$ admits a solution set $J \subseteq \Icof{I}\cap \Weq$ at $I$.
  \end{description}
  Then $\ctg{C}$ is a combinatorial model category where
  $\Weq$ is the class of weak equivalences while $I$ is a set of
  generating cofibrations and $J$ is a set of generating trivial cofibrations.
\end{theorem}

We refer to~\cite{beke:shhomc} for an extensive discussion of Theorem~\ref{thm:smith}. In the original statement, \sfour only requires the existence of a solution set, without any inclusion condition. The present version brings a minor simplification in the treatment of our particular case.

For the remaining of this section, we assume the hypotheses of Theorem~\ref{thm:smith}.

\begin{lemma}
  {\dueto{Smith}}\label{lem:smith} Suppose there is a class $\mathcal{J}
  \subseteq \Icof{I}\cap \Weq$ such that each
  commuting square
  \begin{center}
   $
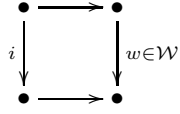
\xymatrix{
\bullet \ar[d]_i \ar[r] & \bullet \ar[d]^{w \in {\mathcal W}} \\ 
\bullet \ar[r] & \bullet
}$
  \end{center}
  admits a factorisation
  \begin{center}
   $
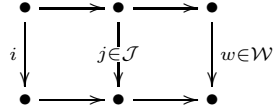
\xymatrix{
\bullet \ar[d]_i \ar[r] & 
\bullet\ar[r] \ar[d]|{j \in {\mathcal J}} &
\bullet \ar[d]^{w \in {\mathcal W}} 
\\ 
\bullet \ar[r] & \bullet \ar[r] & \bullet
}$
  \end{center}
  
  Then
  \[ \Icof{\mathcal{J}} = \Icof{I}\cap \Weq \]
\end{lemma}

Lemma~\ref{lem:smith} is a key step in the proof of Theorem~\ref{thm:smith}. This is Lemma~1.8 in~\cite{beke:shhomc}, where a complete proof is given, based again on the small object argument combined with an induction step.

\begin{rem}\label{rem:J} We have
  \[ \Icof{J} = \Icof{I}\cap \Weq \]
  by Lemma~\ref{lem:smith}, so in particular
  \[ \Ifib{J} = \Ifib{(\Icof{I}\cap \Weq)} \]
  by remark \ref{rem:inj-cell-cof}.
\end{rem}

\begin{lemma}\label{lem:triv-fibs}
$\Ifib{I} = \Ifib{J}\cap \Weq$.
\end{lemma}

\begin{proof} ``$\subseteq$'' \ Since $\Ifib{I} \subseteq
\Weq$, by \stwo, we need to show that $\Ifib{I}
\subseteq \Ifib{J}$. Let $j \in J$, $f \in \Ifib{I}$ and suppose $f \circ u = v \circ u$ for some $u$ and $v$.
The small object argument produces a factorisation $f = f'' \circ f'$ with $f'
\in \Icof{J}$ and $f'' \in \Ifib{J}$, so there
are $p$ and $q$ such that the following diagram commutes (the existence of $q$ is a consequence of Remark~\ref{rem:J}):

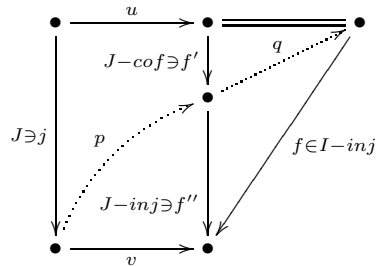
\begin{center}
  $\xygraph{ 
!{<0cm,0cm>;<1cm,0cm>:<0cm,1cm>::}
!{(0,0) }*+{\bullet}="1"
!{(0,-3) }*+{\bullet}="2"
!{(2,0) }*+{\bullet}="3"
!{(2,-1)}*+{\bullet}="4"
!{(2,-3)}*+{\bullet}="5"
!{(4,0)}*+{\bullet}="6"
"1":"2"_{J \ni j}
"1":"3"^u
"2":"5"_v
"3":"4"_{J-cof \ni f^\prime}
"4":"5"_(.7){J-inj \ni f^{\prime\prime}}
"3":@{=}"6"
"6":"5"^{f \in I-inj}
"2":@/^.35cm/@{.>}"4"^p
"4":@{.>}"6"^q
}$
\end{center}

``$\supseteq$'' Let $f \in \Ifib{J} \cap \Weq$. The
small object argument produces a factorisation

\begin{center}
  $\xygraph{ 
!{<0cm,0cm>;<1cm,0cm>:<0cm,1cm>::}
!{(1.25,0) }*+{\bullet}="1"
!{(0,-1) }*+{\bullet}="2"
!{(1.25,-2) }*+{\bullet}="3"
"1":"2"_{I-cof \ni f^\prime}
"1":"3"^{f \in (J-inj) \cap {\mathcal W}}
"2":"3"_{I-inj \ni f^{\prime\prime}}
}$
\end{center}

so $f' \in \Icof{I} \cap \Weq$ by \sone. On the other hand $f \in \Ifib{(\Icof{I}\cap \Weq)}$ by Remark~\ref{rem:J}, so $f \in \Ifib{I}$ by the retract argument (see Proposition~\ref{prop:retracts}).
\end{proof}

\begin{proofthm}{\em \ref{thm:smith}.}
Let $\mathcal{C} \deq \Icof{I}$ \ and $\mathcal{F}\deq \Ifib{J}$. It readily follows that
$\Weq$, $\mathcal{C}$ and $\mathcal{F}$ are the constituent classes of
a model structure on $\ctg{C}$:

\begin{itemizeminus}
  \item \mone holds by hypothesis;
  \item \mtwo  holds by hypothesis for $\Weq$, by construction for $\mathcal{C}$ and $\mathcal{F}$;
  \item as for \mthree, consider a commutative square:
  \begin{center}
   $\xygraph{ 
!{<0cm,0cm>;<1cm,0cm>:<0cm,1cm>::}
!{(0,0) }*+{\bullet}="1"
!{(0,-1.2) }*+{\bullet}="2"
!{(1.2,0) }*+{\bullet}="3"
!{(1.2,-1.2) }*+{\bullet}="4"
"1":"2"_{I-cof \ni c}
"3":"4"^{f \in J-inj}
"1":"3"
"2":"4"
}$
  \end{center}
  If $c \in \Weq$ then this square admits a lift by Remark~\ref{rem:J}. On the other hand, if $f \in \Weq$ then this square admits a lift by Lemma~\ref{lem:triv-fibs};
\item \mfour holds because the factorisations are constructed using the small object argument and have the required properties by Lemma~\ref{lem:triv-fibs} and Remark~\ref{rem:J}, respectively.
\end{itemizeminus}
Therefore $\ctg{C}$ is a combinatorial model category by Remark~\ref{rem:J}.
\end{proofthm}

\section{Higher dimensional categories}\label{sec:omegacat}

This section is devoted to a brief review of higher dimensional categories, here defined as globular sets with structure.
 
\subsection{Globular sets}\label{subsec:globular sets}

Let $\glob$ be the small category whose objects are integers $0,1,\ldots$, and whose morphisms are generated by $\cosce{n},\cotge{n}:n\to n{+}1$ for $n\in\NN$, subject to the following equations:
 \begin{eqnarray*}
          \cosce{n+1}\circ\cosce{n} & = & \cotge{n+1}\circ\cosce{n}, \\
          \cosce{n+1}\circ\cotge{n} & = & \cotge{n+1}\circ\cotge{n}.
        \end{eqnarray*}
These equations imply that there are exactly two morphisms from $m$ to $n$ if $m<n$, none if $m>n$, and only the identity if $m=n$.
\begin{defn}\label{defin:globset}
  A {\em globular set} is a presheaf on $\glob$.
\end{defn}
In other words, a globular set is a functor from
$\opp{\glob}$ to $\sets$. Globular sets and natural transformations
form a category $\globset$. If $X$ is a globular set, we denote by $X_n$ the image of $n\in\NN$ by $X$; members of $X_n$ are called {\em $n$-cells}. By defining $\sce{n}=X(\cosce{n})$ and $\tge{n}=X(\cotge{n})$, we get {\em source} and {\em target} maps
\begin{displaymath}
  \sce n,\tge n:X_{n{+}1}\to X_n.
\end{displaymath}
More generally, whenever $m>n$, one defines
\begin{eqnarray*}
  \SCE nm & = & \sce n\circ\cdots\circ\sce{m{-}1},\\
  \TGE nm & = & \tge n\circ\cdots\circ\tge{m{-}1},
\end{eqnarray*}
so that $\SCE nm$ and $\TGE nm$ are maps from $X_m$ to $X_n$.
Let us call two $n$-cells $x$, $y$ {\em parallel} whenever $n=0$, or $n>0$ and
\begin{eqnarray*}
  \sce{n{-}1}(x) & = & \sce{n{-}1}(y),\\
  \tge{n{-}1}(x) & = & \tge{n{-}1}(y).
\end{eqnarray*}
We write $x\para y$ whenever $x$, $y$, are parallel cells:
        
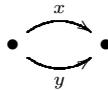
\begin{displaymath}
          \begin{xy}
               \xymatrix{
               \bullet\ar @/^2ex/[r]^{x}\ar @/_2ex/[r]_{y}
               &
            \bullet }
          \end{xy}        
        \end{displaymath}

We will need a few additional notations about globular sets:
\begin{itemizeminus}
\item if $u$ is an $n{+}1$-cell, we write $u : x \to y$ whenever $\sce nu=x$ and $\tge nu=y$, in which case $x\para y$;
\item if $m > n$ and $u$ is an $m$-cell, we write $u : x \To n y$ whenever $\SCE nm(u)=x$ and $\TGE nm(u)=y$. Here again $x$, $y$ are parallel $n$-cells;
\item we write $u\Cons n v$ if $u : x \To n y$ and $v : y \To n  z$ for some $m$-cells $u$, $v$ and $n$-cells $x$, $y$, $z$;
\item if $n > 0$ and $u$ is an $n$-cell, we write $\Sce u$ for $\SCE 0n(u)$ and $\Tge u$ for $\TGE 0n(u)$, so that we get $u : \Sce u \To 0 \Tge u$. 
\end{itemizeminus}

\subsection{Strict $\omega$-categories}\label{subsec:highdimcat}

A {\em strict $\omega$-category} is a globular set $C$ endowed with operations of composition and units, satisfying the laws of associativity, units and interchange, as follows: 
\begin{itemizeminus}
\item if $u$, $v$ are $m$-cells such that $u \Cons n v$, we write $u \Comp n v$ for the $n$-\emph{composition} of $u$ with $v$ (in diagrammatic order);
\item if $x$ is an $n$-cell, we write $\unit x : x \to x$ for the corresponding $n{+}1$-\emph{dimensional unit};
\item if $x$ is an $n$-cell and $m > n$, we write $\Unit m x$ for the corresponding $m$-\emph{dimensional unit}. We also write $\Unit n x$ for $x$;
\item if $m > n > p$, we write $u \Comp p v$ for $\Unit m u \Comp p v$ whenever $u : x \To p y$ is an $n$-cell and $v : y \To p z$ is an $m$-cell;
\item similarly, we write $u \Comp p v$ for $u \Comp p \Unit m v$ whenever $u : x \To p y$ is an $m$-cell and $v : y \To p z$ is an $n$-cell.
\end{itemizeminus}
If $m > n$, the following identities hold for any $m$-cells $u \Cons n v \Cons n w$ and for any $m$-cell $u : x \To n y$:
\[
(u \Comp n v) \Comp n w = u \Comp n (v \Comp n w), \qquad
\Unit m x \Comp n u = u = u \Comp n \Unit m y.
\]
If $m > n > p$, the following identities hold for any $m$-cells $u \Cons n u'$ and $v \Cons n v'$ such that $u \Cons p v$ (so that $u' \Cons p v'$), for any $n$-cells $x \Cons p y$, and for any $p$-cell $z$:
\[
(u \Comp n u') \Comp p (v \Comp n v') = (u \Comp p v) \Comp n (u' \Comp p v'), \qquad
\Unit m x \Comp p \Unit m y = \Unit m {x \Comp p y}, \qquad
\Unit m {\Unit n z} = \Unit m z.
\]
An $\omega$-\emph{functor} is a morphism of globular sets preserving compositions and units. Thus, $\omega$-categories and $\omega$-functors build the category $\ocat$, which is our main object of study.

The forgetful functor $U:\ocat\to\globset$ is finitary monadic~\cite{batanin:mongcw} and $\globset$ is a topos of presheaves on a small category: therefore $\ocat$ is complete and cocomplete. On the other hand, the left adjoint to $U$ takes a globular set to the {\em free} $\omega$-category it generates. In particular, consider $Y:\glob\to\globset$ the Yoneda embedding: we get, for each $n$, a representable globular set $Y(n)=\Hom{\glob}{-}{n}$. 
\begin{defn}\label{def:nglobe}
  For $n\geq 0$, the {\em $n$-globe $\OO n$} is the free $\omega$-category generated by $Y(n)$. 
\end{defn}

Notice that $\OO n$ has exactly two non-identity $i$-cells for $i<n$, exactly one non-identity $n$-cell, and no non-identity cells in dimensions $i>n$.

\begin{proposition}\label{prop:locfp}
  $\ocat$ is locally finitely presentable.
\end{proposition}
 
\begin{proof}
It is a general fact that the representable objects $Y(n)$ are {\em finitely presentable}. Because $U$ preserves filtered colimits, all $n$-globes are finitely presentable objects in $\ocat$.
\end{proof}

\subsection{Shift construction}\label{subsec:shift}
The following construction will prove essential in defining the functor $\cnx$  of Section~\ref{subsec:connect} below. Thus, given an $\omega$-category $C$ and two $0$-cells $x$, $y$ in it, we define a new $\omega$-category $\HOM x y$ as follows:
\begin{itemizeminus}
\item there is an $n$-cell $\Sht u$ in $\HOM x y$ for each $n{+}1$-cell $u : x \To 0 y$;
\item for any $n{+}1$-cells $u, v : x \To 0 y$ and for any $n{+}2$-cell $w : u \to v$, we have $\Sht w : \Sht u \to \Sht v$ in $\HOM x y$;
\item $n$-composition is defined by $\Sht u \Comp n \Sht v = \Sht{u \Comp{n{+}1} v}$ whenever $u \Cons{n{+}1} v$;
\item $m$-dimensional units are defined by $\Unit m {\Sht u} = \Sht{\Unit {m{+}1} u}$.
\end{itemizeminus}
The verification of the axioms of $\omega$-categories is straightforward. We shall use some additional operations described below. For any $0$-cells $x, y, z$, we get:
\begin{itemizeminus}
\item a \emph{precomposition} $\omega$-functor $u \act {-} : \HOM y z \to \HOM x z$ for each 1-cell $u : x \to y$, defined by $u \act \Sht v = \Sht{u \Comp 0 v}$; 
\item a \emph{postcomposition} $\omega$-functor ${-} \act v : \HOM x y \to \HOM x z$ for each 1-cell $v : y \to z$, defined by $\Sht u \act v = \Sht{u \Comp 0 v}$;
\item a \emph{composition $\omega$-bifunctor} ${-} \COMP {-} : \HOM x y \times \HOM y z \to \HOM x z$, defined by $\Sht u \COMP \Sht v = \Sht{u \Comp 0 v}$.
\end{itemizeminus}

\section{The folk model structure}\label{sec:folkmodel}

The first step is to consider, for each $n$, the globular set $\partial Y(n)$ having the same cells as $Y(n)$ except for removing the unique $n$-cell. Thus $\partial Y(n)$ generates an $\omega$-category $\DO n$, the {\em boundary} of the $n$-globe, and we get an inclusion $\omega$-functor
\begin{displaymath}
  \ii n:\DO n\to \OO n.
\end{displaymath}
Notice that, for each $n$, we get a pushout:

\begin{equation}
\xygraph{ 
!{<0cm,0cm>;<1cm,0cm>:<0cm,1cm>::}
!{(0,0) }*+{\DO n}="ul"
!{(0,-1.5) }*+{\OO n}="dl"
!{(2,0) }*+{\OO n}="ur"
!{(2,-1.5) }*+{\DO{n{+}1}}="dr"
!{(1.5,-1.)}*+{}="pusha"
!{(1.7,-1.)}*+{}="pushb"
!{(1.5,-1.2)}*+{}="pushc"
"ul":"dl"_{\ii n}
"ul":"ur"^{\ii n}
"ur":"dr"
"dl":"dr"
"pusha":@{-}"pushb"
"pusha":@{-}"pushc"
}
\label{eq:pushout}
\end{equation}

The rest of this section is devoted to the construction of a
combinatorial model structure on $\ocat$ where \[\I \; \deq \; \{\ii n \;
| \; n \in \NN\}\] is a set of generating cofibrations.

\subsection{$\I$-injectives}\label{subsec:iinj}

Notice that an $\omega$-functor $f : X \to Y$ in $\Tfb$ can equivalently be characterised as verifying the following conditions:
\begin{itemizeminus}
\item for any 0-cell $y$ in $Y$, there is a 0-cell $x$ in $X$ such that $f \, x = y$;
\item for any $n$-cells $x \para x'$ in $X$ and for any $v : f \, x \to f \, x'$ in $Y$, there is $u : x \to x'$ in $X$ such that $f \, u = v$.
\end{itemizeminus} 

\begin{lemma}\label{lemma:surjectivity}
An $\omega$-functor $f : X \to Y$ in $\Tfb$ satisfies the following properties:
\begin{itemizeminus}
\item for any $n$-cell $y$ in $Y$, there is an $n$-cell $x$ in $X$ such that $f \, x = y$;
\item for any $n$-cells $y \para y'$ in $Y$, there are $n$-cells $x \para x'$ in $X$ such that $f \, x = y$ and $f \, x' = y'$.
\end{itemizeminus}
\end{lemma}

\subsection{Omega-equivalence} \label{subsec:equivalence}

Our definition of weak equivalences is based on two notions:
reversible cells and $\omega$-equivalence between parallel cells. 
These notions are defined by mutual coinduction.
\begin{defn}
For any $n$-cells $x \para y$ in some $\omega$-category:
\begin{itemizeminus}
\item we say that $x$ and $y$ are $\omega$-\emph{equivalent}, and we write $x \eqv y$, if there is a \emph{reversible} $n{+}1$-cell $u : x \rto y$;
\item we say that the $n{+}1$-cell $u : x \to y$ is \emph{reversible}, and we write $u : x \rto y$, if there is an $n{+}1$-cell $\inv u : y \to x$ such that $u \Comp n \inv u \eqv \unit x$ and $\inv u \Comp n u \eqv \unit y$.
\end{itemizeminus}
Such a $\inv u$ is called a \emph{weak inverse} of $u$.
\end{defn}
Notice that there is no base case in such a definition. Hence, we get infinite trees of cells of increasing dimension. 

We now establish the first properties of reversible cells and $\omega$-equivalence.
\begin{lemma}\label{lemma:functor}
For any $\omega$-functor $f : X \to Y$ and for any $u : x \rto x'$ in $X$, we have $f \, u : f \, x \rto f \, x'$ in $Y$. Hence, $f$ preserves~$\eqv$.
\end{lemma}
\begin{proof}
  Suppose that $x$, $x'$ are $n$-cells with $u:x\rto x'$. By definition, there
  is an $n{+}1$-cell $\inv u:x'\to x$ such that $u\Comp n\inv u\eqv\unit x$ and $\inv u\Comp n u\eqv\unit{x'}$, whence reversible $n{+}2$-cells $v:u\Comp n\inv u\rto\unit x$ and $v':\inv u\Comp n u\rto\unit{x'}$. Now, by coinduction, $f\, v:f\, u\Comp n f\,\inv u\rto\unit{f\,x}$ and $f\, v':f\, \inv u\Comp n f\, u\rto \unit{f\, x'}$. Therefore $f \, u : f \, x \rto f \, x'$.
\end{proof}

\begin{proposition}\label{prop:congruence}
The relation $\eqv$ is an $\omega$-congruence. More precisely:
\begin{enumerateroman}
\item For any $n$-cell $x$, we get $\unit x : x \rto x$. Hence, $\eqv$ is reflexive.\label{item:reflex}
\item For any reversible $n{+}1$-cell $u : x \rto y$, we get $\inv u : y \rto x$. Hence, $\eqv$ is symmetric.\label{item:sym}
\item For any reversible $n{+}1$-cells $u : x \rto y$ and $v : y \rto z$, we get $u \Comp n v : x \rto z$. Hence, $\eqv$ is transitive.\label{item:trans}
\item For any $n$-cells $x, y, z$, and for any $u : x \to y$, $s,t:y\To n z$ and $v : s \rto t$ we get $u \Comp n v : u \Comp n s \rto u \Comp n t$. There is a similar property for postcomposition. Hence, $\eqv$ is compatible with compositions.\label{item:compatible}
\end{enumerateroman}
\end{proposition}
\begin{proof}
For~(\ref{item:reflex}), the proof is by coinduction, whereas~(\ref{item:sym}) follows immediately from the definition. Let $x$, $y$, $z$, $u$, $v$, $s$ and $t$ as in~(\ref{item:compatible}), and consider $f$, the precomposition $\omega$-functor $u\act{-}:\HOM y z\to \HOM x z$ (Section~\ref{subsec:shift}). As $v:s\rto t$, we easily get $\Sht v:\Sht s\rto\Sht t$, so that Lemma~\ref{lemma:functor} applies  and $f\, \Sht v:\Sht s\rto\Sht t$, whence  $u \Comp n v : u \Comp n s \rto u \Comp n t$. The same holds for postcomposition. As for~(\ref{item:trans}), suppose that $u:x\rto y$ and $v:y\rto z$. By definition, there are $n{+}1$-cells $\inv u:y\to u$ and $\inv v:z\to y$ together with reversible $n{+}2$-cells $w:u\Comp n \inv u\rto \unit x$ and $t:v\Comp n\inv v\rto\unit y$. By using the compatibility property~(\ref{item:compatible}) just established, we get $u\Comp n v\Comp n\inv v\Comp n\inv u\eqv u\Comp n\unit y\Comp n \inv u= u\Comp n\inv u$. Also $u\Comp n\inv u\eqv\unit x$. By coinduction, transitivity holds in dimension $n{+}1$, whence $u\Comp n v\Comp n\inv v\Comp n\inv u\eqv \unit x$. Likewise $\inv v\Comp n\inv u\Comp n u\Comp n v\eqv \unit z$. Therefore $u\Comp n v:x\rto z$.
\end{proof}
\sssec
There is a convenient notion of weak uniqueness, related to $\omega$-equivalence.
\begin{defn}\label{def:weakunique}
A condition $\cal C$ defines a \emph{weakly unique} cell $u : x \to y$ if we have $u \eqv u'$ for any other $u' : x \to y$ satisfying $\cal C$.
\end{defn}

A less immediate, but crucial result is the following ``weak division'' property.
\begin{lemma}\label{lemma:weakdiv}
Any reversible $1$-cell $u : x \rto y$ satisfies the \emph{left division property}:
\begin{itemizeminus}
\item For any $1$-cell $w : x \to z$, there is a weakly unique $1$-cell $v : y \to z$ such that $u \Comp 0 v \eqv w$.
\item For any $1$-cells $s, t : y \to z$ and for any $2$-cell $w : u \Comp 0 s \to u \Comp 0 t$, there is a weakly unique $2$-cell $v : s \to t$ such that $u \Comp 0 v \eqv w$.
\item More generally, for all $n > 0$, for any parallel $n$-cells $s, t : y \To 0 z$ and for any $n{+}1$-cell $w : u \Comp 0 s \to u \Comp 0 t$, there is a weakly unique $n{+}1$-cell $v : s \to t$ such that $u \Comp 0 v \eqv w$.
\end{itemizeminus}
Similarly, $u : x \rto y$ satisfies the \emph{right division property}.
\end{lemma}

In fact, this also applies to any reversible $2$-cell $u : x \rto y$, seen as a reversible $1$-cell in the $\omega$-category $\HOM {\Sce u} {\Tge u}$.

\begin{proof} We have a weak inverse $\inv u : y \rto x$ and some reversible 2-cell $r : \inv u \Comp 0 u \rto \unit y$.
\begin{itemizeminus}
\item In the first case, we have $u \Comp 0 v \eqv w$ if and only if $v \eqv \inv u \Comp 0 w$. 
\item In the second case, $u \Comp 0 v \eqv w$ implies $(r \Comp 0 s) \Comp 1 v = (\inv u \Comp 0 u \Comp 0 v) \Comp 1 (r \Comp 0 t) \eqv (\inv u \Comp 0 w) \Comp 1 (r \Comp 0 t)$ by interchange and compatibility. By left division by $r \Comp 0 s$ (first case), this condition defines a weakly unique $v$. Hence, we get weak uniqueness for left division by $u$. Moreover, this condition implies $\inv u \Comp 0 u \Comp 0 v \eqv \inv u \Comp 0 w$ by right division by $r \Comp 0 t$ (first case), from which we get $u \Comp 0 v \eqv w$ by weak uniqueness applied to $\inv u$.
\item The general case (for left and right division) is proved in the same way by induction on $n$.
\end{itemizeminus}
\end{proof}

\subsection{$\omega$-Weak equivalences}\label{subsec:weq}

If we replace equality by $\omega$-equivalence in the definition of $\I$-injectives, we get \weqs.

\begin{defn}\label{def:weq}
An $\omega$-functor $f : X \to Y$ is an \weq{} whenever it satisfies the following conditions:
\begin{enumerateroman}
\item for any 0-cell $y$ in $Y$, there is a 0-cell $x$ in $X$ such that $f \, x \eqv y$;
\item for any $n$-cells $x \para x'$ in $X$ and for any $v : f \, x \to f \, x'$ in $Y$, there is $u : x \to x'$ in $X$ such that $f \, u \eqv v$.
\end{enumerateroman}
We write $\Weq$ for the class of \weqs.
\end{defn}

\begin{rem}\label{rem:S2}
As equality implies $\omega$-equivalence (Proposition~\ref{prop:congruence}), we have 
\begin{displaymath}
  \Tfb \subseteq \Weq,
\end{displaymath}
which is exactly condition~\stwo of Theorem~\ref{thm:smith}.
\end{rem}

\sssec
We first remark that $\omega$-equivalences are {\em weakly injective}, in the sense of the following Lemma.
\begin{lemma}\label{lemma:weakinj}
If $f : X \to Y$ is in $\Weq$, then $x \eqv x'$ for any $x \para x'$ in $X$ such that $f \, x \eqv f \, x'$ in $Y$.
\end{lemma}
\begin{proof}
Let $f\in\Weq$ and $x$, $x'$ parallel $n$-cells such that $f\, x\eqv f\, x'$. There are $n{+}1$-cells $u:f\, x\to f\, x'$ and $\inv u:f\, x'\to f\, x$ such that $u\Comp n \inv u\eqv \unit{f\, x}$ and $\inv u\Comp n u\eqv \unit{f\, x'}$. Because $f$ is a \weq, we get $n{+}1$-cells $v:x\to x'$ and $\inv v:x'\to x$ such that $f\, v\eqv u$ and $f\, \inv v\eqv \inv u$. By using Proposition~\ref{prop:congruence},(\ref{item:trans}) and~(\ref{item:compatible}), and the preservation of compositions and units by $f$,
\begin{eqnarray*}
  f(v\Comp n \inv v) & \eqv & u\Comp n \inv u\\
                      & \eqv & f(\unit x)
\end{eqnarray*}
By coinduction, $v\Comp n\inv v\eqv \unit x$ and likewise $\inv v\Comp n v\eqv \unit{x'}$, whence $x\eqv x'$.
\end{proof}
\sssec
The ``3 for 2'' property states that whenever two $\omega$-functors out of $f$, $g$ and $h=g\circ f$ are \weqs, then so is the third. So there are really three statements, that we shall address separately.
\begin{lemma}\label{lemma:compos}
  Let $f:X\to Y$ and $g:Y\to Z$ be \weqs. Then $g\circ f:X\to Z$
  is in $\Weq$.

\end{lemma}
\begin{proof}
  Suppose that $f:X\to Y$, $g:Y\to Z$ are \weqs{} and let $h=g\circ f$. If $z$ is a $0$-cell in $Z$, there is a $0$-cell $y$ in $Y$ such that $g\, y\eqv z$, and a $0$-cell $x$ in $X$ such that $f\, x\eqv y$. By Lemma~\ref{lemma:functor}, $h\, x\eqv g\, y$, and by Proposition~\ref{prop:congruence},(\ref{item:trans}), $h\, x\eqv z$. Now, let $x$, $x'$ be two parallel $n$-cells in $X$ and $w:h\, x\to h\, x'$ be an $n{+}1$-cell in $Z$. There is a $v:f\, x\to f\, x'$ such that $g\, v\eqv w$ and a $u:x\to x'$ such that $f\, u\eqv v$. By Lemma~\ref{lemma:functor} and Proposition~\ref{prop:congruence},(\ref{item:trans}) again, we get $h\, u\eqv w$ and we are done. 
\end{proof}

\begin{lemma}\label{lemma:rightinv}
  Let $f:X\to Y$, $g:Y\to Z$ be $\omega$-functors and suppose that $g$
  and $g\circ f$ are \weqs. Then $f$ is in $\Weq$.

\end{lemma}
\begin{proof}
  Let $f$, $g$ and $h=g\circ f$ such that $g\in\Weq$ and $h\in \Weq$. Let $y$ be a $0$-cell in $Y$, and $z=g\, y$. There is a $0$-cell $x$ in $X$ such that $h\, x\eqv z$. By Lemma~\ref{lemma:weakinj}, $f\, x\eqv y$. Likewise, let $x$, $x'$ be parallel $n$-cells in $X$ and $v:f\, x\to f\, x'$ an $n{+}1$-cell in $Y$. We get $g\, v:h\, x\to h\, x'$, therefore there is an $n{+}1$-cell $u:x\to x'$ in $X$ such that $h\, u\eqv g\, v$. By Lemma~\ref{lemma:weakinj}, $f\, u\eqv v$ and we are done.
\end{proof}

The remaining part of the 3-for-2 property for $\Weq$ is significantly
harder to show and will be addressed in Section~\ref{subsec:gluing}.
\sssec
\begin{lemma}\label{lemma:retracttransf}
The class $\Weq$ is closed under retract and transfinite composition.
\end{lemma}
\begin{proof}
 The closure under retracts follows immediately from the definition, by using Lemma~\ref{lemma:functor}.

As for the closure under transfinite composition, let $\alpha>0$ be an ordinal, viewed as a category with a unique morphism $\beta\to\gamma$ for each pair $\beta\leq\gamma$ of ordinals $<\alpha$, and $X:\alpha\to\ocat$ a functor, preserving colimits. We denote by $w_\beta^\gamma$ the morphism $X(\beta\to\gamma):X(\beta)\to X(\gamma)$, and by $(\overline X,w_\beta)$ the colimit of the directed system $(X(\beta),w_\beta^\gamma)$. Suppose that each $w_\beta^{\beta+1}$ belongs to $\Weq$. We need to show that $w_0:X(0)\to\overline X$ is still a \weq. We first establish that for each $\beta<\alpha$, $w_0^{\beta}\in\Weq$, by induction on $\beta$:
 \begin{itemizeminus}
 \item if $\beta=0$, $w_0^\beta$ is the identity on $X(0)$, thus belongs to $\Weq$;
\item if $\beta$ is a successor ordinal, $\beta=\gamma+1$ and $w_0^\beta=w_\gamma^{\gamma+1}\circ w_0^\gamma$. By induction, $w_0^\gamma\in\Weq$, and by hypothesis $w_\gamma^{\gamma+1}\in\Weq$, hence the result, by composition;
\item if $\beta$ is a limit ordinal, $\beta=\sup_{\gamma<\beta}\gamma$. Let $n>0$, $(x,y)$ a pair of parallel $n{-}1$ cells in $X(0)$ and $u:w_0^\beta(x)\to w_0^\beta(y)$ an $n$-cell in $X(\beta)$. Because $X$ preserves colimits, there is already a $\gamma<\beta$ and an $n$-cell $v$ in $X(\gamma)$ such that $v:w_0^\gamma(x)\to w_0^\gamma(y)$ and $w_\gamma^\beta(v)=u$. By the induction hypothesis, $w_0^\gamma$ is a \weq, and there is a $z:x\to y$ in $X(0)$ such that $w_0^\gamma(z)\eqv v$. By composing with $w_\gamma^\beta$, we get $w_0^\beta(z)\eqv u$. The same argument applies to the case $n=0$, so that $w_0^\beta\in\Weq$.
 \end{itemizeminus}
Now we complete the proof by induction on $\alpha$ itself: if $\alpha$ is a successor ordinal, then $\alpha=\beta+1$ and $w_0$ is $w_0^\beta$, hence belongs to $\Weq$, as we just proved. If $\alpha$ is a limit ordinal, we reproduce the argument of the limit case above, using again the fact that $w_0^\beta$ is a \weq{} for any $\beta<\alpha$.
\end{proof}

\begin{corollary}\label{cor:retracttrans}
$\Cof \cap \Weq$ is closed under retract and transfinite composition.
\end{corollary}

\subsection{Cylinders}\label{subsec:connect}

The proofs of condition~\stwo, part of~\sone and~\sthree were directly based on our definitions of generating cofibrations and \weqs. As for the remaining points, we shall need a new construction: to each $\omega$-category $X$ we associate an $\omega$-category $\Cnx X$ whose cells are the {\em reversible cylinders} of $X$. The correspondence $\cnx$ turns out to be functorial and endowed with natural transformations from and to the identity functor. Reversible cylinders are in fact cylinders in the sense of~\cite{metayer:respol} and~\cite{lafontmetayer:polrhm}, satifying an additional reversibility condition. In the present work, ``cylinder'' means ``reversible cylinder'', as the general case will not occur. 

\begin{defn}\label{def:cylinder}
By induction on $n$, we define the notion of $n$-\emph{cylinder} $U : x \cto y$ between $n$-cells $x$ and $y$ in some $\omega$-category:
\begin{itemizeminus}
\item a 0-cylinder $U : x \cto y$ in $X$ is given by a reversible $1$-cell $\Pal U : x \rto y$;
\item if $n > 0$, an $n$-cylinder $U : x \cto y$ in $X$ is given by two reversible 1-cells $\Sce U : \Sce x \rto \Sce y$ and $\Tge U : \Tge x \rto \Tge y$, together with some $n{-}1$-cylinder $\Sht U : \Sht x \act \Tge U \cto \Sce U \act \Sht y$ in the $\omega$-category $\HOM {\Sce x} {\Tge y}$.
\end{itemizeminus}
If $U : x \cto y$ is an $n$-cylinder, we write $\Top\, U$ and $\Bot\, U$ for the $n$-cells $x$ and $y$.
\end{defn}

\begin{center}
$
\xygraph{ 
!{<0cm,0cm>;<1cm,0cm>:<0cm,1cm>::} 
%
%
!{(0,0) }*+{x}="x"
!{(3,0) }*+{x^\flat}="xflat"
!{(5,0)}*+{x^\sharp}="xsharp"
!{(0,-2) }*+{y}="y"
!{(3,-2) }*+{y^\flat}="yflat"
!{(5,-2)}*+{y^\sharp}="ysharp"
%
%
"x":"y"_{U^\natural}
"xflat":"yflat"_{U^\flat}
"xflat":"xsharp"^x
"yflat":"ysharp"_y
"xsharp":"ysharp"^{U^\sharp}
"xsharp":"yflat"_{U^\natural}
}
$
\end{center}

We also write $\Top_X\, U$ and $\Bot_X\, U$ to emphasize the fact that $U$ is an $n$-cylinder in the $\omega$-category $X$. The next step is to show that $n$-cylinders in $X$ are the $n$-cells of a globular set.

\begin{defn}\label{def:stcyl}
By induction on $n$, we define the \emph{source $n$-cylinder} $U : x \cto x'$ and the \emph{target $n$-cylinder} $V : y \cto y'$ of any $n{+}1$-cylinder $W : z \cto z'$ between $n{+}1$-cells $z : x \to y$ and $z' : x' \to y'$:
\begin{itemizeminus}
\item if $n = 0$, then $\Pal U = \Sce W$ and $\Pal V = \Tge W$;
\item if $n > 0$, then $\Sce U = \Sce V = \Sce W$ and $\Tge U = \Tge V = \Tge W$, whereas the two $n{-}1$-cylinders $\Sht U$ and $\Sht V$ are respectively defined as the source and the target of the $n$-cylinder $\Sht W$ in the $\omega$-category $\HOM {\Sce z} {\Tge{z'}}$.
\end{itemizeminus}
In that case, we write $W : U \to V$ or also $\Cto {W : U \to V} z {z'}$.
\end{defn}

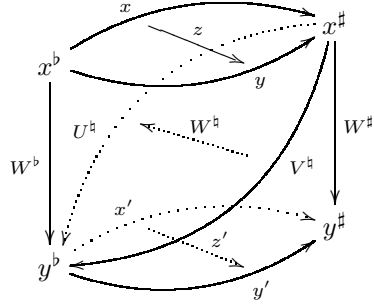
\begin{center}
$
\xygraph{ 
!{<0cm,0cm>;<1.5cm,0cm>:<0cm,.9cm>::} 
!{(0,0) }*+{x^\flat}="xflat"
!{(0,-3) }*+{y^\flat}="yflat"
!{(.1,-2.8) }*+{}="yflatprime" 
!{(.16,-2.71) }*+{}="yflatpprime" 
!{(2.5,0.6)}*+{x^\sharp}="xsharp"
!{(2.44,0.3) }*+{}="xsharpprime" 
!{(2.49,0.4) }*+{}="xsharppprime" 
!{(2.5,-2.4)}*+{y^\sharp}="ysharp"
!{(.8,0.6)}*+{}="x"
!{(1.7,0)}*+{}="y"
!{(.8,-2.4)}*+{}="xp"  
!{(1.7,-3)}*+{}="yp"
!{(1.8,-1.4)}*+{}="domw"  
!{(0.8,-.9)}*+{}="codw"
"xflat":"yflat"_{W^\flat}
"xsharp":"ysharp"^{W^\sharp}
"xflat":@/^.5cm/"xsharp"^(.3)x
"xflat":@/_.55cm/"xsharp"_(.7){y}
"yflat":@{.>}@/^.5cm/"ysharp"^(.3){x^\prime}
"yflat":@/_.55cm/"ysharp"_(.7){y^\prime}
"xsharp":@{.>}@/_.9cm/"yflatprime"_(.7){U^\natural}
"xsharpprime":@/^.9cm/ "yflat"^(.3){V^\natural}
%
"x":"y"^z
"xp":@{.>}"yp"^(.7){z^\prime}
"domw":@{.>}"codw"_{W^\natural}
}
$
\end{center}

\begin{lemma}\label{lemma:boundary}
We have $U \para V$ for any $n{+}1$-cylinder $W : U \to V$. In other words, cylinders form a globular set.
\end{lemma}
\begin{proof}
  By induction on $n$.
\end{proof}

Remark that the 0-source $U$ and the 0-target $V$ of an $n{+}1$-cylinder $W$ are given by $\Pal U = \Sce W$ and $\Pal V = \Tge W$. 
\sssec
We now define {\em trivial cylinders}.
\begin{defn}
By induction on $n$, we define the \emph{trivial $n$-cylinder} $\Triv \, x : x \cto x$ for any $n$-cell $x$:
\begin{itemizeminus}
\item if $n = 0$, then $\Pal{(\Triv \, x)} = \unit x$;
\item if $n > 0$, then $\Sce{(\Triv \, x)} = \unit{\Sce x}$ and $\Tge{(\Triv \, x)} = \unit{\Tge x}$, whereas $\Sht{\Triv \, x}$ is the trivial cylinder $\Triv \Sht x$ in $\HOM {\Sce x} {\Tge x}$.
\end{itemizeminus}
\end{defn}
We also write $\Triv_X\, x$ for $\Triv\, x$ to emphasize the fact that $x$ is an $n$-cell of the $\omega$-category $X$.
The following result is a straightforward consequence of the definition.
\begin{lemma}
We have $\Triv \, x \para \Triv \, y$ for any $n$-cells $x \para y$, and $\Triv \, z : \Triv \, x \to \Triv \, y$ for any $z : x \to y$.
\end{lemma}
More generally, we get the following  notion of {\em degenerate cylinder}:
\begin{defn}
An $n$-cylinder between parallel cells is \emph{degenerate} whenever $n = 0$ or $n > 0$ and its source and target are~trivial.
\end{defn}
Remark that $\Triv \, x \para U \para \Triv \, y$ for any degenerate $n$-cylinder $U : x \cto y$. The next easy lemma gives a more concrete description of degenerate cylinders:
\begin{lemma}
\begin{enumerateroman}
\item For any degenerate $n$-cylinder $U : x \cto y$, we get a reversible $n{+}1$-cell $\Pal U : x \rto y$.
\item Conversely, any reversible $n{+}1$-cell $u : x \rto y$ corresponds to a unique degenerate $n$-cylinder $U : x \cto y$.
\end{enumerateroman}
\end{lemma}
In particular, the trivial $n$-cylinder $\Triv \, x : x \cto x$ is the degenerate $n$-cylinder given by $\Pal{(\Triv \, x)} = \unit x : x \rto x$.
\sssec
Thus, for each $\omega$-category $X$, we have defined a globular set $\Cnx X$ whose $n$-cells are $n$-cylinders in $X$, together with globular morphisms $\Top_X, \Bot_X : \Cnx X \to X$ and $\Triv_X : X \to \Cnx X$ such that $\Top_X \circ \Triv_X = \id_X = \Bot_X \circ \Triv_X$.
\[
\begin{xy}
\xymatrix{& X \ar[ld]_{\id_X} \ar[d]^{\Triv_X} \ar[rd]^{\id_X} & \\ 
X & \Cnx X \ar[l]^{\Top_X} \ar[r]_{\Bot_X} & X}
\end{xy}
\]
 Now we may define compositions of $n$-cylinders in $X$, as well as units, in such a way that the globular set $\Cnx X$ becomes an $\omega$-category: this is done in detail in appendix~\ref{annex:connect} (see also~\cite{metayer:respol} and~\cite{lafontmetayer:polrhm}). Thus, from now on, $\Cnx X$ denotes this $\omega$-category. Likewise, $\Top_X$, $\Bot_X$ and $\Triv_X$ become $\omega$-functors. The following theorem, proved in appendix, summarizes the properties we actually use in the construction of our model structure.
\begin{theorem}\label{thm:connect}
The correspondence $X\mapsto\Cnx X$ is the object part of an endofunctor on $\ocat$, and $\Top,\Bot:\cnx\to \id$, $\Triv:\id\to\cnx$ are natural transformations.
\end{theorem}

In particular, we get $f \, U : f \, x \cto f \, x'$ for any $\omega$-functor $f : X \to Y$ and for any $n$-cylinder $U : x \cto x'$ in $X$.
\sssec
We end this presentation of $n$-cylinders with the following important ``transport'' lemma.
\begin{lemma}\label{lemma:transport}
For any parallel $n$-cylinders $U : x \cto x'$ and $V : y \cto y'$, we have a \emph{topdown transport}:
\begin{enumerateroman}
\item For any $z : x \to y$, there is $z' : x' \to y'$ together with a cylinder $\Cto {W : U \to V} z {z'}$.\label{item:cylinderfill}
\item Such a $z'$ is weakly unique: $z' \eqv z''$ for any $z'' : x' \to y'$ together with a cylinder $\Cto {W' : U \to V} z {z''}$.\label{item:uniquefill}
\item Conversely, there is a cylinder $\Cto {W' : U \to V} z {z''}$ for any $z'' : x' \to y'$ such that $z' \eqv z''$.\label{item:conversefill}
\end{enumerateroman}
Similarly, we have a \emph{bottom up transport}.
\end{lemma}
\begin{proof}
 We proceed by induction on $n$.
\begin{itemizeminus}
 \item If $n=0$, let $U:x\cto x'$ and $V:y\cto y'$ be parallel $0$-cylinders, and a $1$-cell $z:x\to y$. By definition, there are reversible $1$-cells $u:x\to x'$ and $v:y\to y'$. Let $\inv u:x'\to x$ a weak inverse of $u$, and define $z'=\inv u\Comp 0 z\Comp 0 v$. Now $u\Comp 0 z'=u\Comp 0\inv u\Comp 0 z\Comp 0 v$. As $u\Comp 0\inv u\eqv \unit x$, $u\Comp 0 z'\eqv z\Comp 0 v$, by using Proposition~(\ref{prop:congruence}). Whence a reversible $2$-cell $w:z\Comp 0 v\rto u\Comp 0 z'$, that is a reversible $1$-cell, or $0$-cylinder, in the $\omega$-category $\HOM{x}{y'}$. Thus we get a $1$-cylinder $\Cto{W:U\to V}z{z'}$, and~(\ref{item:cylinderfill}) is proved. Suppose now that there is a $z'':x'\to y'$ together with a $1$-cylinder $\Cto{W':U\to V}z{z''}$. It follows that $u\Comp 0 z'\eqv z\Comp 0 v\eqv u\Comp 0 z''$, whence $z'\eqv z''$ by Lemma~\ref{lemma:weakdiv}. This proves~(\ref{item:uniquefill}). Suppose finally that $z''\eqv z'$. We get $u\Comp 0 z''\eqv u\Comp 0 z'\eqv z\Comp 0 v$, and a cylinder $\Cto {W' : U \to V} z {z''}$ as above, which proves~(\ref{item:conversefill}).
\item Suppose that~(\ref{item:cylinderfill}),~(\ref{item:uniquefill}) and~(\ref{item:conversefill}) hold in dimension $n$. Let $U:x\cto x'$, $V:y\cto y'$ parallel $n{+}1$-cylinders and $z:x\to y$ an $n{+}2$-cell. By definition, we have reversible $1$-cells $\Sce U=\Sce V:\Sce x\rto\Sce{x'}$, $\Tge U=\Tge V:\Tge x\rto\Tge{x'}$, together with parallel $n$-cylinders $\Sht U:\Sht x\act\Tge U\cto\Sce U\act\Sht{x'}$ and $\Sht V:\Sht y\act\Tge y\cto\Sce V\act\Sht{y'}$ in $\HOM{\Sce x}{\Tge{y'}}$. Now $\Sht z\act\Tge U:\Sht x\act\Tge U\to \Sht y\act\Tge V$ is an $n{+}1$-cell $\Sht w$ in $\HOM{\Sce x}{\Tge{y'}}$. By the induction hypothesis, we get an $n{+}1$-cell $\Sht{w'}:\Sce U\act\Sht{x'}\to\Sce V\act\Sht{y'}$ and an $n{+}1$-cylinder $\Cto{\Sht{W_0}:\Sht U\to\Sht V}{\Sht w}{\Sht{w'}}$ in $\HOM{\Sce x}{\Tge{y'}}$. By Lemma~\ref{lemma:weakdiv}, there is a $\Sht{z'}:\Sht{x'}\to\Sht{y'}$ such that $\Sht{w'}\eqv\Sce U\act\Sht{z'}$. Thus, part~(\ref{item:conversefill}) of the induction hypothesis gives an $n{+}1$-cylinder $\Cto{\Sht W:\Sht U\to\Sht V}{\Sht z\act\Tge U}{\Sce U\act\Sht{z'}}$. But this defines an $n{+}2$-cylinder $\Cto{W:U\to V}{z}{z'}$, and~(\ref{item:cylinderfill}) holds in dimension $n{+}1$. Moreover, by induction, the above cell $\Sht{w'}$ is weakly unique, and so is $z'$, by Lemma~\ref{lemma:weakdiv}: this gives~(\ref{item:uniquefill}) in dimension $n{+}1$. Finally, if $z''\eqv z'$, $\Sce U\act\Sht{z''}\eqv\Sce U\act\Sht{z'}$ in $\HOM{\Sce x}{\Tge{y'}}$, and the induction hypothesis gives an $n{+}1$-cylinder $\Cto{\Sht{W'}:\Sht U \to\Sht V}{\Sht z\act\Tge U}{\Sce U\act\Sht{z''}}$, whence an $n{+}2$-cylinder $\Cto{W':U\to V}{z}{z''}$, so that~(\ref{item:conversefill}) holds in dimension $n{+}1$.
\end{itemizeminus}
\end{proof}

\begin{corollary}
For each $\omega$-category $X$, $\Top_X, \Bot_X$ are in $\Tfb$ and $\Triv_X$ is in $\Weq$.
\end{corollary}
\begin{proof}
  Let $\Cto Ux{x'}$ and $\Cto Vy{y'}$ be parallel $n$-cylinders in $X$ and $z:\Top_X U\to\Top_X V$ an $n{+}1$-cell. By Lemma~\ref{lemma:transport}, there is an $n{+}1$-cylinder $W:U\to V$ such that $\Top_X W=z$. This proves that $\Top_X$ is in $\Tfb$. Likewise, by bottom up transport, $\Bot_X$ is in $\Tfb$. But $\Tfb\subseteq\Weq$ by \stwo so that $\Top_X$ is a \weq. Now $\Top_X\circ\Triv_X=\id_X$, and by Lemma~\ref{lemma:rightinv}, $\Triv_X\in\Weq$.
\end{proof}

\subsection{Gluing factorization} \label{subsec:gluing}

For any $\omega$-functor $f : X \to Y$, we consider the following
pullback:
\begin{center}
$
\xygraph{ 
!{<0cm,0cm>;<1cm,0cm>:<0cm,1cm>::} 
!{(0,0) }*+{\Glu f}="ul"
!{(0,-1.5) }*+{\Cnx Y}="dl"
!{(2,0) }*+{X}="ur"
!{(2,-1.5) }*+{Y}="dr"
!{(0.5,-0.5)}*+{}="pulla"
!{(0.3,-0.5)}*+{}="pullb"
!{(0.5,-0.3)}*+{}="pullc"
"ul":"dl"_{f'}
"ul":"ur"^{f^*\Top_Y}
"ur":"dr"^{f}
"dl":"dr"_{\Top_Y}
"pulla":@{-}"pullb"
"pulla":@{-}"pullc"
}
$
\end{center}
We write $\lft f : \Glu f \to Y$ for $\Bot \circ f'$, so that the
following diagram commutes:
\begin{center}
$
\xygraph{ 
!{<0cm,0cm>;<1cm,0cm>:<0cm,1cm>::} 
!{(0,0) }*+{\Glu f}="ul"
!{(0,-1.5) }*+{\Cnx Y}="dl"
!{(2,0) }*+{X}="ur"
!{(2,-1.5) }*+{Y}="dr"
!{(-2,0) }*+{X}="ull"
!{(-2,-1.5) }*+{Y}="dll"
!{(0.5,-0.5)}*+{}="pulla"
!{(0.3,-0.5)}*+{}="pullb"
!{(0.5,-0.3)}*+{}="pullc"
"ul":"dl"_{f'}
"ul":"ur"^{f^*\Top_Y}
"ur":"dr"^{f}
"dl":"dr"_{\Top_Y}
"ull":"dll"_f
"ull":"ul"^{\rht f}
"dll":"dl"_{\Triv_Y}
"ull":@/^.75cm/"ur"^{\id_X}
"dll":@/_.75cm/"dr"_{\id_Y}
"pulla":@{-}"pullb"
"pulla":@{-}"pullc"
}
$
\end{center}

Since $\Top_Y$ is in $\Tfb$, so is its pullback $f^*\Top_Y$. By \stwo, $f^*\Top_Y$ is in $\Weq$. As $f^*\Top_Y\circ\rht f=\id_X$, by Lemma~\ref{lemma:rightinv}, $\rht f$ is also a \weq. 

\begin{defn}\label{def:glufact}
The decomposition $f = \lft f \circ \rht f$ is called the \emph{gluing factorization} of $f$.
\[
\begin{xy}
\xymatrix{X \ar[r]^{\rht f} \dar[rr]_{f} & \Glu f \ar[r]^{\lft f} & Y}
\end{xy}
\]
\end{defn}
The above constructions may be described more concretely as follows:
\begin{itemizeminus}
\item an $n$-cell in $\Glu f$ is a pair $(x, U)$ where $x$ is an $n$-cell in $X$ and $U : f \, x \cto y$ is an $n$-cylinder in $Y$; 
\item $\rht f \, x = (x, \Triv \, f \, x)$ for any $n$-cell $x$ in $X$, and $\lft f (x, U) = \Bot \, U = y$ for any $n$-cylinder $U : f \, x \cto y$ in $Y$.
\end{itemizeminus}
The gluing factorization leads to an extremely useful characterization of \weqs.
\begin{proposition}\label{prop:charweq}
An $\omega$-functor $f : X \to Y$ is in $\Weq$ if and only if $\lft f : \Glu f \to Y$ is in $\Tfb$.
\end{proposition}
\begin{proof}
Suppose that $\lft f$ is in $\Tfb$, then it is in $\Weq$ by \stwo; as $\rht f$ is a \weq, so is the composition $f = \lft f \circ \rht f$, by Lemma~\ref{lemma:compos}. Conversely, suppose that $f$ is in $\Weq$, and let us show that $\lft f$ is in $\Tfb$:
\begin{itemizeminus}
\item For any 0-cell $y$ in $Y$, there is a 0-cell $x$ in $X$ such that $f \, x \eqv y$. Hence, we get a reversible $1$-cell $u : f \, x \rto y$ defining a 0-cylinder $U : f \, x \cto y$, so that $(x, U)$ is a 0-cell in $\Glu f$ and $\lft f (x, U) = y$.
\item For any $n$-cells $(x, T) \para (x', T')$ in $\Glu f$, we get parallel $n$-cylinders $T : f \, x \cto y$ and $T' : f \, x' \cto y'$. For any $n{+}1$-cell $w : y \to y'$, Lemma~\ref{lemma:transport}, bottom up direction, gives $v : f \, x \to f \, x'$ together with $\Cto {V : T \to T'} v w$. Since $f$ is in $\Weq$ and $x \para x'$, we get an $n{+}1$-cell $u : x \to x'$ such that $f \, u \eqv v$. By Lemma~\ref{lemma:transport},~(\ref{item:conversefill}), bottom up direction, we get $\Cto {U : T \to T'} {f \, u} w$, so that $(u, U) : (x, T) \to (x', T')$ is an $n{+}1$-cell in $\Glu f$ and $\lft f (u, U) = w$.
\end{itemizeminus}
\end{proof}

\begin{corollary} \label{cor:minimal}
$\Weq$ is the smallest class containing $\Tfb$ which is closed under composition and right inverse.
\end{corollary}

It is now possible to prove the remaining part of condition 3-for-2
for $\Weq$.
\begin{lemma}\label{lemma:leftinv}
  If $f : X \to Y$ and $h = g \circ f : X \to Z$ are in $\Weq$, so is
  $g : Y \to Z$.

\end{lemma}
\begin{proof}
\begin{itemizeminus}
\item For any 0-cell $z$ in $Z$, there is a 0-cell $x$ in $X$ such that $h \, x \eqv z$. So we get $g \, y \eqv z$, where $y = f \, x$.
\item Let $y \para y'$ be $n$-cells in $Y$, and let $w : g \, y \to g \, y'$ be an $n{+}1$-cell in $Z$.
\begin{itemizeminus}
\item By Proposition~\ref{prop:charweq}, $\lft f$ is in $\Tfb$, so that Lemma~\ref{lemma:surjectivity} applies, and we get $x \para x'$ in $X$ and parallel $n$-cylinders $T : f \, x \cto y$ and $T' : f \, x' \cto y'$.
\item By Theorem~\ref{thm:connect}, we get parallel $n$-cylinders $g \,T : h \, x \cto g \, y$ and $g \,T' : h \, x' \cto g \, y'$.
\item By Proposition~\ref{prop:charweq}, $\lft h$ is in $\Tfb$ and we get $u : x \to x'$ together with $\Cto {U : g \, T \to g \, T'} {h \, u} w$.
\item By Lemma~\ref{lemma:transport},~(\ref{item:cylinderfill}) we get $v : y \to y'$ together with $\Cto {V : T \to T'} {f \, u} v$.
\item By Theorem~\ref{thm:connect}, we get $\Cto {g \, V : g \, T \to g \, T'} {h \, u} {g \, v}$.
\item By Lemma~\ref{lemma:transport},~(\ref{item:uniquefill}), we get $g \, v \eqv w$.
\end{itemizeminus}
\end{itemizeminus}
\end{proof}

\subsection{Immersions}\label{subsec:immersions}

In order to complete the proof of condition~\sthree, we introduce a new class of $\omega$-functors.

\begin{defn}\label{def:immersions}
An \emph{immersion} is an $\omega$-functor $f : X \to Y$ satisfying the following three conditions:
\begin{description}
\item[(Z1)] there is a retraction $g : Y \to X$ such that $g \circ f = \id_X$;
\item[(Z2)] there is an $\omega$-functor $h : Y \to \Cnx Y$ such that $\Top_Y \circ h = f \circ g$ and $\Bot_Y \circ h = \id_Y$;
\item[(Z3)] $h \circ f = \Triv_Y \circ f$. In other words, $h$ is trivial on $f(X)$.
\end{description}
\[
\begin{xy}
\xymatrix{X \ar[r]^{f} \dar[rr]_{\id_X} & Y \ar[r]^{g} & X}
\end{xy}
\qquad
\begin{xy}
\xymatrix{X \ar[d]_{f} & Y \ar[l]_{g} \ar[d]^{h} \uar[rd]^{\id_Y} \\ 
Y & \Cnx Y \ar[l]^{\Top_Y} \ar[r]_{\Bot_Y} & Y}
\end{xy}
\qquad
\begin{xy}
\xymatrix{X \ar[r]^{f} \ar[d]_{f} & Y \ar[d]^{h} \\
Y \ar[r]_{\Triv_Y} & \Cnx Y}
\end{xy}
\]
We write $\Imm$ for the class of immersions.
\end{defn}

Notice that, by naturality of $\Triv$, condition~\zthree can be replaced by the following one:
\begin{description}
\item[(Z3')] $h \circ f = \Cnx f \circ \Triv_X$.
\end{description}

The gluing construction of the previous section yields a characterization of immersions by a lifting property.
 
\begin{lemma}\label{lemma:charimm}
An $\omega$-functor $f : X \to Y$ is an immersion if and only if there is an $\omega$-functor $k : Y \to \Glu f$ such that $k \circ f = \rht f$ and $\lft f \circ k = \id_Y$.
\[ 
\begin{xy}
\xymatrix{X \ar[r]^{\rht f} \ar[d]_{f} & \Glu f \ar[d]^{\lft f} \\ 
Y \ear[ru]^{k} \ar[r]_{\id_Y} & Y}
\end{xy}
\]
\end{lemma}
\begin{proof}
Let $f:X\to Y$, and suppose that there is a $k:Y\to\Glu f$ satisfying the above lifting property. Define $g=f^*\Top_Y\circ k$ and $h=f'\circ k$. We get $g\circ f=f^*\Top_Y\circ k\circ f=f^*\Top_Y\circ \rht f=\id_X$, hence~\zone. Also $\Top_Y\circ h=\Top_Y\circ f'\circ k=f\circ f^*\Top_Y\circ k=f\circ g$ and $\Bot_Y\circ h=\Bot_Y\circ f'\circ k=\lft f\circ k=\id_Y$, hence~\ztwo. Finally $h\circ f=f'\circ k\circ f= f'\circ\rht f=\Triv_Y\circ f$, hence~\zthree.

Conversely, suppose that $f:X\to Y$ is an immersion, and let $g$, $h$ satify the conditions of Definition~\ref{def:immersions}. By~\ztwo, $\Top_Y\circ h=f\circ g$, so that the universal property of $\Glu f$ yields a unique $k:Y\to\Glu f$ such that $f^*\Top_Y\circ k=g$ and $f'\circ k=h$. Thus $\lft f\circ k=\Bot_Y\circ f'\circ k=\Bot_Y\circ h=\id_Y$, by~\ztwo. Now $f^*\Top_Y\circ k\circ f=g\circ f=\id_X=f^*\Top_Y\circ \rht f$ and $f'\circ k\circ f=h\circ f=\Triv_Y\circ f$ by~\zthree so that $f'\circ k\circ f=f'\circ\rht f$: by the universal property of $\Glu f$, this gives $k\circ f=\rht f$, and we are done.
\end{proof}

\begin{corollary}\label{cor:trivcofimm}
$\Cof \cap \Weq \subseteq \Imm$.
\end{corollary}
\begin{proof}
Suppose that $f:X\to Y$ belongs to $\Cof\cap\Weq$. As $f\in\Weq$, by Proposition~\ref{prop:charweq}, $\lft f\in\Tfb$. Now $f\in Cof$ has the left lifting property with respect to $\lft f$, so that there is a $k$ such  that $k\circ f=\rht f$ and $\lft f\circ k=\id_Y$. By Lemma~\ref{lemma:charimm}, $f$ is an immersion. 
\end{proof}

\begin{lemma}\label{lemma:immweq}
$\Imm \subset \Weq$.
\end{lemma}
\begin{proof}
Suppose that $f : X \to Y$ is an immersion, and let $g$, $h$ as in Definition~\ref{def:immersions}:
\begin{itemizeminus}
\item For any 0-cell $y$ in $Y$, we get $h \, y : f \, x \cto y$ where $x = g \, y$. Hence, we get $\Pal{(h \, y)} : f \, x \rto y$, so that $f \, x \eqv y$.
\item For any $n$-cells $x \para x'$ in $X$ and for any $v : f \, x \to f \, x'$ in $Y$, we have $h \, v : f \, u \cto v$ where $u = g \, v : x \cto x'$. By \zthree, the cylinder $h \, v : \Triv \, f \, x \to \Triv \, f \, x'$ is degenerate. Hence, we get $\Pal{(h \, v)} : f \, u \rto v$, so that $f \, u \eqv v$.
\end{itemizeminus}
\end{proof}

\begin{lemma}\label{lemma:pushoutimm}
$\Imm$ is closed under pushout.
\end{lemma}
\begin{proof}
Let $f : X \to Y$ be an immersion, $i : X \to X'$ an $\omega$-functor
and $f':X'\to Y'$ the pushout of $f$ by $i$:

\begin{center}
$
\xygraph{ 
!{<0cm,0cm>;<1cm,0cm>:<0cm,1cm>::}
!{(0,0) }*+{X}="ul"
!{(0,-1.5) }*+{Y}="dl"
!{(2,0) }*+{X^\prime}="ur"
!{(2,-1.5) }*+{Y^\prime}="dr"
!{(1.5,-1.)}*+{}="pusha"
!{(1.7,-1.)}*+{}="pushb"
!{(1.5,-1.2)}*+{}="pushc"
"ul":"dl"_f
"ul":"ur"^i
"ur":"dr"^{f^\prime}
"dl":"dr"_j
"pusha":@{-}"pushb"
"pusha":@{-}"pushc"
}
$
\end{center}

Since $f$ is an immersion, we have $g : Y \to X$ and $h : Y \to \Cnx
Y$ satisfying conditions~\zone to~\zthree. By universality of the
pushout and by \zthreep, we get $g' : Y' \to X'$ and $h' : Y' \to
\Cnx{Y'}$ such that the following diagrams commute:

\[
\xygraph{ 
!{<0cm,0cm>;<1cm,0cm>:<0cm,1cm>::}
!{(0,0) }*+{X}="ul"
!{(0,-1.5) }*+{Y}="dl"
!{(2,0) }*+{X^\prime}="ur"
!{(2,-1.5) }*+{Y^\prime}="dr"
!{(0,-3) }*+{X}="ddl"
!{(2,-3) }*+{X^\prime}="ddr"
!{(1.5,-1.)}*+{}="pusha"
!{(1.7,-1.)}*+{}="pushb"
!{(1.5,-1.2)}*+{}="pushc"
"ul":"dl"_f
"ul":"ur"^i
"ur":"dr"^{f^\prime}
"dl":"dr"_j
"dl":"ddl"_g
"dr":@{-->}"ddr"^{g'}
"ddl":"ddr"_i
"ul":@/_.8cm/"ddl"_{\id_X}
"ur":@/^.8cm/"ddr"^{\id_{X'}}
"pusha":@{-}"pushb"
"pusha":@{-}"pushc"
}
\qquad
\xygraph{ 
!{<0cm,0cm>;<1cm,0cm>:<0cm,1cm>::}
!{(0,0) }*+{X}="ul"
!{(0,-1.5) }*+{Y}="dl"
!{(2,0) }*+{X^\prime}="ur"
!{(2,-1.5) }*+{Y^\prime}="dr"
!{(0,-3) }*+{\Cnx Y}="ddl"
!{(2,-3) }*+{\Cnx {Y'}}="ddr"
!{(-1,-1.5) }*+{\Cnx X}="dll"
!{(3,-1.5) }*+{\Cnx {X'}}="drr"
!{(1.5,-1.)}*+{}="pusha"
!{(1.7,-1.)}*+{}="pushb"
!{(1.5,-1.2)}*+{}="pushc"
"ul":"dl"_f
"ul":"ur"^i
"ur":"dr"^{f^\prime}
"dl":"dr"_j
"dl":"ddl"_h
"dr":@{-->}"ddr"^{h'}
"ddl":"ddr"_{\Cnx j}
"ul":@/_.3cm/"dll"_{\Triv_X}
"dll":@/_.3cm/"ddl"_{\Cnx f}
"ur":@/^.3cm/"drr"^{\Triv_{X'}}
"drr":@/^.3cm/"ddr"^{\Cnx {f'}}
"pusha":@{-}"pushb"
"pusha":@{-}"pushc"
}
\]

Finally, conditions~\zone to~\zthree for $g'$ and $h'$ follow from conditions~\zone to~\zthree for $g$ and $h$.
\end{proof}

\begin{corollary} \label{cor:thepushout}
$\Cof \cap \Weq$ is closed under pushout.
\end{corollary}
\begin{proof}
  Let $f\in\Cof\cap\Weq$ and $f'$ a pushout of $f$. By Corollary~\ref{cor:trivcofimm}, $f$ is an immersion, and so is $f'$ by Lemma~\ref{lemma:pushoutimm}. By Lemma~\ref{lemma:immweq}, $f'$ is a \weq. Now $\Cof$ is stable by pushout, so that $f'\in\Cof$. Hence $f'\in\Cof\cap\Weq$ and we are done.
\end{proof}

\subsection{Generic squares} \label{subsec:generic}

By Yoneda's Lemma, for each $n$, the functor $X\mapsto X_n$, from $\ocat$ to $\sets$ is represented by the $n$-globe $\OO n$. Thus, to each $n$-cell $x$ of $X$ corresponds a unique $\omega$-functor 
\begin{displaymath}
  \sng x:\OO n\to X.
\end{displaymath}
Moreover, for any pair $x$, $x'$ of $n$-cells in $X$, the condition of parallelism $x\para x'$ is equivalent to  $\sng x\circ\ii n=\sng{x'}\circ\ii n$. By the pushout square~(\ref{eq:pushout}) mentioned at the beginning of Section~\ref{sec:folkmodel}, we get a unique $\omega$-functor 
\begin{displaymath}
  \pair{x}{x'}:\DO{n{+}1}\to X.
\end{displaymath}
associated to any pair $x$, $x'$ of parallel $n$-cells. This applies
in particular to the case where $x =x'= o$, the unique proper $n$-cell
of $\OO n$. The corresponding $\omega$-functor is denoted by $\oo n =
\pair o o : \DO{n{+}1} \to \OO n$. Since $\ocat$ is locally presentable,
there is a factorization $\oo n = \pp n \circ \kk n$ with $\pp n \in \Tfb$ and $\kk n \in \Cof$.
\[
\begin{xy}
\xymatrix{\DO{n{+}1} \ar[r]^{\kk n} \dar[rr]_{\oo n} & \PP n \ar[r]^{\pp n} & \OO n}
\end{xy}
\]
Now by composition of $\kk n$ with both $\omega$-functors $\OO n\to
\DO{n{+}1}$ of the pushout~(\ref{eq:pushout}), we get $\jj n,\jjj
n:\OO n\to\PP n$ such that the following diagram commutes:
\begin{center}
$
\xygraph{ 
!{<0cm,0cm>;<1cm,0cm>:<0cm,1cm>::}
!{(0,0) }*+{\DO n}="ul"
!{(0,-1.5) }*+{\OO n}="dl"
!{(2,0) }*+{\OO n}="ur"
!{(2,-1.5) }*+{\DO{n{+}1}}="dr"
!{(4,-1.5) }*+{\PP n}="drr"
!{(6,-1.5) }*+{\OO n}="drrr"
!{(1.5,-1.)}*+{}="pusha"
!{(1.7,-1.)}*+{}="pushb"
!{(1.5,-1.2)}*+{}="pushc"
"ul":"dl"_{\ii n}
"ul":"ur"^{\ii n}
"ur":"dr"
"dl":"dr"
"dr":"drr"^{\kk n}
"drr":"drrr"^{\pp n}
"ur":@/^.4cm/"drr"^(.7){\jj n}
"ur":@/^.54cm/"drrr"^{\id_{\OO n}}
"dl":@/_.4cm/"drr"_(.75){\jjj n}
"dl":@/_.9cm/"drrr"_{\id_{\OO n}}
"pusha":@{-}"pushb"
"pusha":@{-}"pushc"
}
$
\end{center}

The following definition singles out an important part of the above diagram.
\begin{defn}
The \emph{generic $n$-square} is the following commutative square:
\[
\begin{xy}
\xymatrix{\DO n \ar[d]_{\ii n} \ar[r]^{\ii n} & \OO n \ar[d]^{\jj n} \\
\OO n \ar[r]_{\jjj n} & \PP n}
\end{xy}
\]
\end{defn}

\begin{rem}\label{rem:solutionset}
Notice that $\pp n$ is in $\Tfb$, hence in $\Weq$, and that $\pp
n\circ\jj n=\id_{\OO n}$. Therefore $\jj n\in\Weq$, by
Lemma~\ref{lemma:rightinv}. On the other hand ${\ii n} \in
\Icof{I}$. Since $\Icof{I}$ is stable under composition and pushout, we
have $\jj n\in\Icof{I}$
\end{rem}

The next result characterizes the relation of $\omega$-equivalence in terms of suitable factorizations.
\begin{lemma}\label{lemma:charomegaeq}
For any $n$-cells $x \para x'$ in $X$, the following conditions are equivalent:
\begin{enumerateroman}
\item \label{item:eqv}  $x \eqv x'$;
\item \label{item:factory} there is an $\omega$-category $Y$ and $\omega$-functors $k:\DO{n{+}1}\to Y$, $p:Y\to \OO n$ and $q:Y\to X$ such that $p\in\Tfb$ and the following diagram commutes:
\[
\begin{xy}
\xymatrix{& \DO{n{+}1} \ar[ld]_{\oo n} \ear[d]^{k} \ar[rd]^{\pair x {x'}} & \\
\OO n & \ear[l]^{p} Y \ear[r]_{q} & X}
\end{xy};
\]
\item \label{item:factorp} There is an $\omega$-functor $q:\PP n\to X$ such that the following diagram commutes:
\[
\begin{xy}
\xymatrix{& \DO{n{+}1} \ar[ld]_{\oo n} \ar[d]^{\kk n} \ar[rd]^{\pair x {x'}} & \\
\OO n & \ar[l]^{\pp n} \PP n \ear[r]_{q} & X}
\end{xy}.
\]
\end{enumerateroman}
\end{lemma}
\begin{proof}
If $x \eqv x'$, there is a reversible $n{+}1$-cell $u : x \rto x'$ which defines a degenerate $n$-cylinder $U : x \cto x'$. We get $\Triv_X \, x \para U$, whereas $\Top_X \, \Triv_X \, x = \Bot_X \, \Triv_X \, x = \Top_X \, U = x$ and $\Bot_X \, U = x'$, so that the following diagrams commute:
\[
\begin{xy}
\xymatrix{\DO{n{+}1} \ar[r]^{\oo n} \ar[rd]^{\pair x x} \ar[d]_{\pair {\Triv_X \, x} U} & \OO n \ar[d]^{\sng x} \\ 
\Cnx X \ar[r]_{\Top_X} & X}
\end{xy}
\qquad
\begin{xy}
\xymatrix{\DO{n{+}1} \ar[rd]^{\pair x {x'}} \ar[d]_{\pair {\Triv_X \, x} U} \\ 
\Cnx X \ar[r]_{\Bot_X} & X}
\end{xy}
\]
Let $f = \sng x$. By universality of $\Glu f$, we get $k : \DO{n{+}1} \to \Glu f$ such that the following diagram commutes:
\[
\begin{xy}
\xymatrix{\DO{n{+}1} \uar[rr]^{\oo n} \ear[r]_{k} \dar[rd]_{\pair {\Triv_X \, x} U} & \Glu f \ar[r]_{f^*\Top_X} \ar[d]_{f'} & \OO n \ar[d]^{f} \\ 
& \Cnx X \ar[r]_{\Top_X} & X}
\end{xy}
\]
The desired factorizations are given by $Y = \Glu f$, $p = f^*\Top_X$ and $q = \lft f = \Bot_X \circ f'$. Hence,~(\ref{item:eqv}) implies~(\ref{item:factory}).

Conversely, if we assume~(\ref{item:factory}), then $k$ gives us two $n$-cells $y \para y'$ in $Y$ such that $p \, y = p \, y'$, $q \, y = x$ and $q \,y' = x'$. Hence, we get $y \eqv y'$ by Lemma~\ref{lemma:weakinj} applied to $p$, and $x \eqv x'$ by Lemma~\ref{lemma:functor} applied to $q$.

On the other hand, if we assume~(\ref{item:factory}), then $k$ factors through $\kk n$ by the left lifting property, and so does $\pair x {x'}$. Hence~(\ref{item:factory}) implies~(\ref{item:factorp}). Conversely,~(\ref{item:factorp}) is just a special case of~(\ref{item:factory}). 
\end{proof}

\sssec

We now turn to a new characterization of \weqs.
\begin{proposition}\label{prop:charweqbis}
An $\omega$-functor $f : X \to Y$ is an \weq{} if and only if
any commutative square whose left arrow is $\ii n$ and whose right
arrow is $f$ factors through the generic $n$-square.
\[
\begin{xy}
\xymatrix{\DO n \uar[rr]^{} \ar[r]_{\ii n} \ar[d]_{\ii n} & \OO n \ar[d]^{\jj n} \ear[r] & X \ar[d]^{f} \\ 
\OO n \ar[r]^{\jjj n} \dar[rr]_{} & \PP n \ear[r] & Y}
\end{xy}
\]
\end{proposition}
\begin{proof} Let $f : X \to Y$ be an \weq, and consider a commutative diagram
  \begin{displaymath}
    \begin{xy}
      \xymatrix{\DO n\ar[d]_{\ii n}\ar[r] & X\ar[d]^f\\
                \OO n\ar[r] & Y}
    \end{xy}.
  \end{displaymath}
We show that it factors through the generic $n$-square: 
\begin{itemizeminus}
\item If $n = 0$, the commutative square is given by some 0-cell $y$ in $Y$:
\[
\begin{xy}
\xymatrix{\ZERO \ar[r]^{} \ar[d] & X \ar[d]^{f} \\ 
\SNG \ar[r]_{\sng y} & Y}
\end{xy}
\]
Since $f$ is in $\Weq$, there is a 0-cell $x$ in $X$ such that $f \, x \eqv y$, and by the previous lemma, we get $q : \PP 0 \to Y$ such that $q \circ \kk 0 = \pair {f \, x} y$, which means that the following diagram commutes:
\[
\begin{xy}
\xymatrix{\ZERO \uar[rr]^{} \ar[r] \ar[d] & \SNG \ar[d]^{\jj 0} \ear[r]_{\sng x} & X \ar[d]^{f} \\ 
\SNG \ar[r]^{\jjj 0} \dar[rr]_{\sng y} & \PP 0 \ear[r]^{q} & Y}
\end{xy}
\]
\item If $n > 0$, the commutative square is given by $n{-}1$-cells $x \para x'$ in $X$ and some $n$-cell $v : f \, x \to f \, x'$ in $Y$:
\[
\begin{xy}
\xymatrix{\DO n \ar[r]^{\pair x {x'}} \ar[d]_{\ii n} & X \ar[d]^{f} \\ 
\OO n \ar[r]_{\sng v} & Y}
\end{xy}
\]
Since $f$ is in $\Weq$, there is $u : x \to x'$ in $X$ such that $f \, u \eqv v$, and by Lemma~\ref{lemma:charomegaeq}, we get $q : \PP n \to Y$ such that $q \circ \kk n = \pair {f \, u} v$, which means that the following diagram commutes:
\[
\begin{xy}
\xymatrix{\DO n \uar[rr]^{\pair x {x'}} \ar[r]_{\ii n} \ar[d]_{\ii n} & \OO n \ar[d]^{\jj n} \ear[r]_{\sng u} & X \ar[d]^{f} \\ 
\OO n \ar[r]^{\jjj n} \dar[rr]_{\sng v} & \PP n \ear[r]^{q} & Y}
\end{xy}
\]
\end{itemizeminus}
The converse is proved by the same argument.
\end{proof}

\sssec

\begin{corollary} \label{cor:solutionset}

The class $\Weq$ of \weqs{} admits the solution set $\J = \{ \jj n | n
\in \NN \}$.

\end{corollary}

We may finally state the central result of this work:

\begin{theorem}\label{th:main}
$\ocat$ is a combinatorial model category. Its class of
weak equivalences is the class $\Weq$ of \weqs{} while $\I$ and $\J$ are
the sets of generating cofibrations and generating trivial
cofibrations, respectively.
\end{theorem}

\begin{proof} $\ocat$ is locally presentable by proposition
  \ref{prop:locfp} while
  \begin{itemizeminus}

  \item condition $\sone$ holds by lemma \ref{lemma:compos}, lemma
    \ref{lemma:rightinv}, lemma \ref{lemma:leftinv} and lemma
    \ref{lemma:retracttransf};

  \item condition $\stwo$ holds by remark \ref{rem:S2};
  \item condition $\sthree$ holds by corollary \ref{cor:retracttrans}
    and corollary \ref{cor:thepushout};
  \item condition $\sfour$ holds by corollary \ref{cor:solutionset}.
  \end{itemizeminus}

\end{proof}

\begin{rem}

By corollary \ref{cor:minimal}, the model structure of theorem \ref{th:main} is \emph{left-determined}
in the sense of \cite{rosickytholen:leftdm}.

\end{rem}

\section{Fibrant and cofibrant objects}\label{sec:cofibrant}

Recall that, given a model category $\ctg{C}$,  an object $X$ of $\ctg{C}$ is {\em fibrant} if the unique morphism $!_X:X\to 1$ is a fibration. Dually, $X$ is {\em cofibrant} if the unique morphism $0_X:0\to X$ is a cofibration. Now $X$ is fibrant if and only if, for any trivial cofibration $f:Y\to Z$ and any $u:Y\to X$, there is a $v:Z\to X$ such that $v\circ f=u$: in fact, this implies that $!_X:X\to 1$ has the right-lifting property with respect to trivial cofibrations.
\begin{displaymath}
    \begin{xy}
      \xymatrix{Y\ar[r]^u\ar[d]_f & X\ar[d]^{!_X}\\
                Z\ar[r]_{!_Z}\ear[ur]_v & 1}
    \end{xy}
  \end{displaymath}
Likewise, $X$ is cofibrant if and only if for any trivial fibration $p:Y\to Z$ and any morphism $u:X\to Z$ there is a lift $v:X\to Y$ such that $p\circ v=u$.
\begin{displaymath}
    \begin{xy}
      \xymatrix{0\ar[r]^{0_Y}\ar[d]_{0_X} &Y\ar[d]^p \\
               X\ar[r]_u\ear[ur]^v & Z }
    \end{xy}
  \end{displaymath}

\subsection{Fibrant $\omega$-categories}\label{subsec:fibocat}

In the folk model structure on $\ocat$, the characterization of fibrant objects is the simplest possible, as shown by the following result.
\begin{proposition}\label{prop:fibrant}
  All $\omega$-categories are fibrant.
\end{proposition}
\begin{proof}
  Let $X$ be an $\omega$-category, $f:Y\to Z$ a trivial cofibration, and $u:Y\to X$ an $\omega$-functor. By Corollary~\ref{cor:trivcofimm}, $f$ is an immersion. In particular there is a retraction $g:Z\to Y$ such that $g\circ f=\id_Y$. Let $v=u\circ g$. We get $v\circ f= u\circ g\circ f=u$. Hence $X$ is fibrant.
  \begin{displaymath}
    \begin{xy}
      \xymatrix{Y\ar[r]^u\ar[d]^f & X\\
                Z\ear[ur]_v\uar[u]^g & }
    \end{xy}
  \end{displaymath}
\end{proof}

\subsection{Cofibrant $\omega$-categories}\label{subsec:cofibocat}

Our understanding of the cofibrant objects in $\ocat$ is based on an appropriate notion of {\em freely generated} $\omega$-category: notice that the free $\omega$-categories in the sense of the adjunction between $\ocat$ and $\globset$ are not sufficient, as there are too few of them. We first describe a process of generating free cells in each dimension. In dimension $0$, we just have a set $S_0$ and no operations, so that $S_0$ generates $\free S_0=S_0$. In dimension $1$, given a graph
\begin{displaymath}
  \begin{xy}
    \xymatrix{\free S_0 & \doubl{\sce 0}{\tge 0} S_1}
  \end{xy}
\end{displaymath}
where $\free S_0$ is the set of vertices, $S_1$ the set of edges, and $\sce 0$, $\tge 0$ are the source and target maps, there is a free category generated by it: 
\begin{displaymath}
  \begin{xy}
    \xymatrix{\free S_0 & \doubl{\sce 0}{\tge 0}\free S_1}
  \end{xy}.
\end{displaymath}
Now suppose that we add a new set $S_2$ together with a graph
\begin{displaymath}
  \begin{xy}
    \xymatrix{\free S_1 & \doubl{\sce 1}{\tge 1} S_2}
  \end{xy}
\end{displaymath}
satisfying the boundary conditions $\sce 0\circ\sce 1=\sce 0\circ\tge 1$ and $\tge 0\circ \sce 1=\tge 0\circ\tge 1$. What we get is a {\em computad}, a notion first introduced in~\cite{street:limicf}, freely generating a $2$-category
 \begin{displaymath}
  \begin{xy}
    \xymatrix{\free S_0 & \doubl{\sce 0}{\tge 0}\free S_1 & \doubl{\sce 1}{\tge 1}\free S_2}
  \end{xy}.
\end{displaymath} 
This pattern has been extended to all dimensions, giving rise to {\em $n$-computads}~\cite{power:ncatpt} or {\em polygraphs}~\cite{burroni:higdwp,burroni:highdw}. More precisely, let $\nglobset n$ (resp.\ $\ncat n$) denote the category of $n$-globular sets (resp.\ $n$-categories), we get a commutative diagram
\begin{equation}
  \begin{xy}
    \xymatrix{\ncat{(n{+}1)}\ar[r]\ar[d]_{\TT_n} & \nglobset{(n{+}1)}\ar[d]\\
               \ncat n\ar[r]  & \nglobset n}     
  \end{xy}
\label{eq:ncatnglob}
\end{equation}
where the  horizontal arrows are the obvious forgetful functors and the vertical arrows are truncation functors, removing all $n{+}1$-cells. On the other hand, let $\ncatpl n$ be the category defined by the following pullback square:
\begin{equation}
\xygraph{ 
!{<0cm,0cm>;<1cm,0cm>:<0cm,1cm>::} 
!{(0,0) }*+{\ncatpl{n}}="ul"
!{(0,-1.5) }*+{\ncat n}="dl"
!{(2,0) }*+{ \nglobset{(n{+}1)}}="ur"
!{(2,-1.5) }*+{\nglobset n}="dr"
!{(0.5,-0.5)}*+{}="pulla"
!{(0.3,-0.5)}*+{}="pullb"
!{(0.5,-0.3)}*+{}="pullc"
"ul":"dl"_{\TTT_n}
"ul":"ur"^{}
"ur":"dr"^{}
"dl":"dr"_{}
"pulla":@{-}"pullb"
"pulla":@{-}"pullc"
}
\label{eq:ncatplus}
\end{equation}
From~(\ref{eq:ncatnglob}), we get a unique functor $R_n:\ncat{(n{+}1)}\to \ncatpl n$ such that $\TTT_nR_n=\TT_n$, where $\TT_n$ and $\TTT_n$ are the truncation functors appearing in~(\ref{eq:ncatnglob}) and~(\ref{eq:ncatplus}) repectively. Now the key to the construction of polygraphs is the existence of a left-adjoint $L_n:\ncatpl n\to\ncat{(n{+}1)}$ to this $R_n$. Concretely, if $X$ is an $n$-category and $S_{n{+}1}$ a set of  $n{+}1$-cells attached to $X$ by
\begin{equation}
  \begin{xy}
    \xymatrix{X_n & \doubl{\sce n}{\tge n} S_{n{+}1}}
  \end{xy}
\label{eq:graph}
\end{equation}
satisfying the boundary conditions, then $L_n$ builds an $(n{+}1)$-category whose explicit construction is given in \cite{metayer:cofohc}. Here we just mention the following property of $L_n$: let $X^+$ be an object of $\ncatpl n$ given by an $n$-category
\begin{displaymath}
  \begin{xy}
\xymatrix{X_0 & \cdots \doubl{\sce 0}{\tge 0} & X_n\doubl{\sce{n-1}}{\tge{n-1}}}
  \end{xy}
\end{displaymath}
and a graph~(\ref{eq:graph}) then the $n{+}1$-category $L_n X^+$ has the same $n$-cells as $\TTT_nX^+$. In other words, there is a set of $n{+}1$-cells $\free S_{n{+}1}$ such that $L_n X^+$ has the form
\begin{displaymath}
  \begin{xy}
\xymatrix{X_0 & \cdots\doubl{\sce 0}{\tge 0} & X_n\doubl{\sce{n-1}}{\tge{n-1}} & \free S_{n{+}1}\doubl{\sce n}{\tge n} }
  \end{xy}
\end{displaymath}
\begin{defn}\label{def:npolygraph}
{\em $n$-polygraphs} are defined inductively by the following conditions:
\begin{itemizeminus}
 \item  a {\em $0$-polygraph} is a set $S^{(0)}$;
 \item  an {\em $n{+}1$-polygraph} is an object $S^{(n+1)}$ of $\ncatpl n$ such that $\TTT_n S^{(n+1)}$ is of the form $L_n S^{(n)} $ where $S^{(n)}$ is an $n$-polygraph.
\end{itemizeminus}
Likewise, a {\em polygraph} $S$ is a sequence $(S^{(n)})_{n\in\NN}$ of $n$-polygraphs such that, for each $n$, $\TTT_nS^{(n+1)}=L_nS^{(n)}$.  
\end{defn}
The pullback~(\ref{eq:ncatplus}) gives a notion of morphisms for $\ncatpl n$, which, by induction, determines a notion of morphism between $n$-polygraphs, and polygraphs. Thus we get a category $\pol$ of polygraphs and morphisms. By Definition~\ref{def:npolygraph} and the abovementioned property of $L_n$, we may see a polygraph $S$ as an infinite diagram of the following shape:
\begin{equation}
  \begin{xy}
    \xymatrix{S_0\ar[d]&S_1\ar[d]\Doubld&S_2\ar[d]\Doubld&S_3\ar[d]\Doubld
    &\cdots\Doubld\\
         \free S_0&\free S_1\doubl{}{}&\free S_2\doubl{}{}&
      \free S_3 \doubl{}{}
         &\cdots\doubl{}{}}   
  \end{xy}.
\label{eq:polygraph}
\end{equation}
In~(\ref{eq:polygraph}), each $S_n$ is the set of generators of the $n$-cells, the oblique double arrows represent the attachment of new $n$-cells on the previously defined $n{-}1$-category, thus defining an object $X^+$ of $\ncatpl{(n{-}1)}$, whereas $\free S_n$ is the set of $n$-cells in $L_{n-1}X^+$. The bottom line of ~(\ref{eq:polygraph}) displays the {\em free $\omega$-category} generated by the polygraph $S$. This defines a functor $Q:S\mapsto \free S$ from $\pol$ to $\ocat$, which is in fact a left-adjoint. A detailed description of the right-adjoint $P:X\mapsto P(X)$ from $\ocat$ to $\pol$ is given in~\cite{metayer:respol}. 
\sssec
It is now possible to state the main result of this section: 
\begin{theorem}\label{thm:cofibrant}
  An $\omega$-category is cofibrant if and only if it is freely generated by a polygraph.
\end{theorem}
Suppose that $X$ is freely generated by a polygraph $S$, $p:Y\to Z$ is a trivial fibration and $u:X\to Z$ is an $\omega$-functor.  It is easy to build a lift $v:X\to Y$ such that $p\circ v=u$ dimensionwise by using the universal property of the functors $L_n$.
\begin{displaymath}
  \begin{xy}
    \xymatrix{ & Y\ar[d]^p\\
              \free S\ar[r]_u\ear[ur]^v & Z}
  \end{xy}
\end{displaymath}
Thus freely generated $\omega$-categories are cofibrant. The proof of the converse is much harder, and is the main purpose of~\cite{metayer:cofohc}. The problem reduces to the fact that the full subcategory of $\ocat$ whose objects are free on polygraphs is Cauchy complete, meaning that its idempotent morphisms split. 
\sssec
The results of~\cite{metayer:respol} may be revisited in the framework of the folk model structure on $\ocat$. In fact, a resolution of an $\omega$-category $X$ by a polygraph $S$ is a trivial fibration $\free S\to X$, hence a cofibrant replacement of $X$. Notice that for each $\omega$-category $X$, the counit of the adjunction between $\pol$ and $\ocat$ gives an $\omega$-functor
\begin{displaymath}
  \epsilon_X:QP(X)\to X
\end{displaymath}
which is a trivial fibration, and defines the {\em standard resolution} of $X$.

\section{Model structure on $\ncat{n}$}\label{sec:ncat}

In this section, we show that the model structure on $\ocat$ we just
described yields a model structure
on the category $\ncat{n}$ of (strict, small) $n$-categories
for each integer $n\geq 1$. In particular, we recover the known
folk model structures on $\cat$~\cite{joyaltierney:strscs} and
$\twocat$~\cite{lack:quitwo,lack:quibic}.

\sssec

Let $n\geq 1$ be a fixed integer. There is an inclusion functor
\begin{displaymath}
  \II:\ncat{n}\to\ocat
\end{displaymath}
which simply adds all necessary unit cells in dimensions $k>n$.
This functor $\II$ has a left adjoint
\begin{displaymath}
  \SH:\ocat\to\ncat{n}.
\end{displaymath}
Precisely, if $X$ is an $\omega$-category and $0\leq k\leq n$,
the $k$-cells of $\SH X$ are exactly those of $X$ for $k<n$,
whereas $(\SH X)_n$ is the quotient of $X_n$ modulo the {\em congruence
generated} by $X_{n{+}1}$. In other words, parallel
$n$-cells $x$, $y$ in $X$ are {\em congruent modulo $X_{n{+}1}$}
if and only if there is a sequence $x_0=x,x_1,\ldots,x_p=y$ of
$n$-cells and a sequence $z_1,\ldots,z_p$ of $n{+}1$-cells such that,
for each $i=1,\ldots,p$ either $z_i:x_{i{-}1}\To n x_i$ or
$z_i: x_i\To n x_{i{-}1}$.

\ms

Notice that the functor $\II$ also has a right adjoint, namely 
the {\em truncation} functor $\TT:\ocat\to\ncat{n}$ which simply forgets
all cells of dimension $k>n$.

\begin{theorem}
  The inclusion functor $\II:\ncat{n}\to\ocat$ creates a model
   structure on $\ncat{n}$, in which the weak equivalences are the $n$-functors
   $f$ such that $\II(f)\in\Weq$, and $(\SH(\ii k))_{k\in\NN}$ is a
   family of generating cofibrations.
\end{theorem}

The general situation is investigated in~\cite{beke:shhtwo}, 
whose proposition~2.3 states 
sufficient conditions for the transport of a model structure
along an adjunction. In our particular case, these
conditions boil down to the following: 
\begin{enumerate}
  \item[\cone] the model structure on $\ocat$ is cofibrantly generated;
  \item[\ctwo] $\ncat{n}$ is locally presentable;
  \item[\cthree] $\Weq$ is closed under filtered colimits in $\ocat$;
  \item[\cfour] $\II$ preserves filtered colimits;
  \item[\cfive]  If $j\in J$ is a generating trivial cofibration of $\ocat$,
            and $g$ is a pushout of $\SH(j)$ in $\ncat{n}$,
            then $\II(g)$ is a weak equivalence in $\ocat$.
\end{enumerate}
Conditions \cone and \ctwo are known already. Condition \cthree 
follows from the definition of weak equivalences and the fact that
the $\omega$-categories $\OO n$ are finitely presentable
objects in $\ocat$. The functor $\II$, being
left adjoint to $\TT$, preserves all colimits, in particular filtered
ones, hence \cfour. 

\sssec

We now turn to the proof of the remaining condition \cfive. 
First remark that 
$\SH\II$ is the identity on $\ncat{n}$, so that
the monad $\GG=\II\SH$ is idempotent and the monad multiplication
$\mu:\GG^2\to \GG$ is the identity.  As a consequence, if $\eta:1\to\GG$
denotes the unit of the monad, for each $\omega$-category $X$
\begin{equation}
  \GG(\eta_X) =\unit{\GG(X)}.
  \label{eq:geta}
\end{equation}
Also, for each $\omega$-functor of the form $u:\GG(X)\to\GG(Y)$,
\begin{equation}
  \GG(u)=u.
  \label{eq:guu}
\end{equation}   
Now let $X$ be an $\omega$-category. For each $k>n$, all $k$-cells
of $\GG(X)$ are units. Therefore, by construction of the connection functor
$\cnx$, all $k$-cells in $\cnx\GG(X)$ are also units, which implies that
$\cnx\GG(X)$ belongs to the image of $\II$, whence
\begin{equation}
  \GG\cnx\GG(X) = \cnx\GG(X).
  \label{eq:gcg}
\end{equation}
We successively get the natural 
transformations:
\begin{displaymath}
  \eta_X : X\to\GG(X),\qquad
  \cnx(\eta_X) : \cnx(X)\to\cnx\GG(X),\qquad
  \GG\cnx(\eta_X) : \GG\cnx(X)\to \GG\cnx\GG(X)=\cnx\GG(X),
\end{displaymath}
by~(\ref{eq:gcg}). Thus $\lambda_X=\GG\cnx(\eta_X)$ yields a natural
transformation
\begin{displaymath}
  \lambda:\GG\cnx \to \cnx\GG.
\end{displaymath}

\begin{lemma}\label{lemma:gimm}
  The monad $\GG$ on $\ocat$ preserves immersions.
\end{lemma}
\begin{proof}
Let $f:X\to Y$ be an immersion. We want to show that
$f'=\GG(f)$ is still an immersion. By Definition~\ref{def:immersions}, there are 
$g:Y\to X$ and $h:Y\to\cnx(Y)$ such that
\begin{eqnarray*}
  g\circ f & = & \id ;\\
  \Top_Y\circ h & = & f\circ g ; \\
   \Bot_Y\circ h & = & \id ;\\
   h\circ f & = & \Triv_Y\circ f.
\end{eqnarray*}
Let $g'=\GG(g):\GG(Y)\to\GG(X)$ and 
$h'=\lambda_Y\circ\GG(h):\GG(Y)\to \cnx\GG(Y)$, it is now sufficient
to check the following equations:
\begin{eqnarray}
  g'\circ f' & = & \id ;
   \label{eq:gid}\\
  \Top_{\GG(Y)}\circ h'& = & f'\circ g';
   \label{eq:topg}\\
  \Bot_{\GG(Y)}\circ h' & = & \id ;
   \label{eq:botg}\\
  h'\circ f' & = & \Triv_{\GG(Y)}\circ f'.
    \label{eq:trivg}
\end{eqnarray}
Equation~(\ref{eq:gid}) is obvious from functoriality. As for~(\ref{eq:topg}),
we first notice that, by naturality of $\Top$, the following diagram
commutes:
\begin{displaymath}
  \begin{xy}
    \xymatrix{\cnx(Y)\ar[d]_{\cnx(\eta_Y)}\ar[r]^{\Top_Y} & Y\ar[d]^{\eta_Y} \\
              \cnx\GG(Y)\ar[r]_{\Top_{\GG(Y)}} & \GG(Y)}
  \end{xy}.
\end{displaymath}
By applying $\GG$ to the above diagram, we get
\begin{displaymath}
  \begin{xy}
    \xymatrix{\GG\cnx(Y)\ar[d]_{\lambda_Y}\ar[r]^{\GG(\Top_Y)}
               & \GG(Y)\ar[d]^{\GG(\eta_Y)} \\
              \cnx\GG(Y)\ar[r]_{\GG(\Top_{\GG(Y)})} & \GG(Y)}
  \end{xy}.
\end{displaymath}
Now, by~(\ref{eq:geta}), $\GG(\eta_Y)=\unit{\GG(Y)}$ and because
$\cnx\GG(Y)=\GG\cnx\GG(Y)$, by~(\ref{eq:guu}),
$\GG(\Top_{\GG(Y)})=\Top_{\GG(Y)}$. Hence
\begin{equation}
  \GG(\Top_Y) = \Top_{\GG(Y)}\circ \lambda_Y.
\end{equation}
Thus
\begin{eqnarray*}
  \Top_{\GG(Y)}\circ h' & = & \Top_{\GG(Y)}\circ\lambda_Y\circ\GG(h),\\
                        & = & \GG(\Top_Y) \circ \GG(h), \\
                        & = & \GG(\Top_Y\circ h), \\
                        & = & \GG(f\circ g), \\
                        & = & \GG(f)\circ\GG(g), \\
                        & = & f'\circ g'.
\end{eqnarray*}
Equations~(\ref{eq:botg}) and~(\ref{eq:trivg}) hold by the same arguments
applied to the natural transformations $\Bot$ and $\Triv$ respectively. 
Hence $\GG(f)$ is an immersion, and we are done.
\end{proof}
\begin{lemma}\label{lemma:pushimm}
  Let $f:X\to Y$ be an immersion, and suppose
  the following square is a pushout in $\ncat{n}$: 
  \begin{displaymath}
\xygraph{ 
!{<0cm,0cm>;<1cm,0cm>:<0cm,1cm>::} 
!{(0,0) }*+{\SH(X)}="ul"
!{(0,-1.5) }*+{\SH(Y)}="dl"
!{(2,0) }*+{A}="ur"
!{(2,-1.5) }*+{B}="dr"
!{(1.5,-1.)}*+{}="pusha"
!{(1.7,-1.)}*+{}="pushb"
!{(1.5,-1.2)}*+{}="pushc"
"ul":"dl"_{\SH(f)}
"ul":"ur"^{u}
"ur":"dr"^{g}
"dl":"dr"_{v}
"pusha":@{-}"pushb"
"pusha":@{-}"pushc"
}
\end{displaymath}
Then $\II(g)$ is an immersion.
\end{lemma}
\begin{proof}
As $\II$ is left adjoint to $\TT$, it preserves pushouts, and
the following square is a pushout in $\ocat$:
\begin{displaymath}
\xygraph{ 
!{<0cm,0cm>;<1cm,0cm>:<0cm,1cm>::} 
!{(0,0) }*+{\GG(X)}="ul"
!{(0,-1.5) }*+{\GG(Y)}="dl"
!{(2,0) }*+{\II(A)}="ur"
!{(2,-1.5) }*+{\II(B)}="dr"
!{(1.5,-1.)}*+{}="pusha"
!{(1.7,-1.)}*+{}="pushb"
!{(1.5,-1.2)}*+{}="pushc"
"ul":"dl"_{\GG(f)}
"ul":"ur"^{\II(u)}
"ur":"dr"^{\II(g)}
"dl":"dr"_{\II(v)}
"pusha":@{-}"pushb"
"pusha":@{-}"pushc"
}
\end{displaymath}
Now $f$ is an immersion, and so is $\GG(f)$ by Lemma~\ref{lemma:gimm}.
As immersions are closed by pushouts (Lemma~\ref{lemma:pushoutimm}),
$\II(g)$ is also an immersion.  
\end{proof}

\sssec

Now let $j$ be a generating trivial cofibration in $\ocat$, and $g$ a
pushout of $\SH(j)$ in $\ncat{n}$. By Corollary~\ref{cor:trivcofimm}, $j$
is an immersion, so that Lemma~\ref{lemma:pushimm} applies, and
$\II(g)$ is an immersion. By Lemma~\ref{lemma:immweq}, immersions are
weak equivalences, so that $\II(g)\in\Weq$. Hence condition~\cfive
holds, and we are done.

\sssec

In case $n=1$, the weak equivalences of $\ncat{n}$ are exactly the
equivalences of categories, whereas if $n=2$, they are 
the {\em biequivalences} in the sense of~\cite{lack:quitwo}. Moreover,
from the generating cofibrations of $\ocat$ we immediately get
a family of generating cofibrations in $\ncat{n}$, namely the
$n$-functors
\begin{displaymath}
  \SH(\ii{k}):\SH(\DO k)\to\SH(\OO k)
\end{displaymath}
for all $k\in\NN$. By abuse of language, let us denote $\SH(X)=X$ whenever
$X$ is an $\omega$-category of the form $\II(Y)$, that is without
non-identity cells in dimensions $>n$. Likewise, denote $\SH(f)=f$
for each $\omega$-functor $f$ of the form $\II(g)$.
With this convention
\begin{itemizeminus}
  \item for each integer $k\leq n$, $\SH(\ii k)=\ii k$;
  \item $\SH(\ii{n{+}1})$ is the collapsing map
        $\ii{n{+}1}':\DO{n{+}1}\to\OO n$;
  \item for each $k>n{+}1$, $\SH(\ii k)$ is the identity on $\OO n$.
\end{itemizeminus}
Now the right-lifting property with respect to identities is clearly void.
Thus we only need a {\em finite} family of $n{+}2$ generating cofibrations
\begin{displaymath}
  \ii 0,\ldots,\ii n,\ii{n{+}1}'.
\end{displaymath}
If $n=1$ or $n=2$, these are precisely the generating cofibrations
of~\cite{joyaltierney:strscs} and~\cite{lack:quitwo} respectively.
Therefore the corresponding model structures are particular cases of ours.

\appendix

\section{The functor $\cnx$} \label{annex:connect}

The aim of this section is to give a complete proof of Theorem~\ref{thm:connect}. In order to do that, we extend $\omega$-functors to cylinders and we introduce the following operations:
\begin{itemizeminus}
\item \emph{left} and \emph{right action} of cells on cylinders, written $u \act V$ and $U \act v$;
\item \emph{concatenation} of cylinders, written $U \comp V$;
\item \emph{multiplication} of cylinders, written $U \COMP V$;
\item \emph{compositions} of cylinders and the \emph{units}, written $U \Comp n V$ and $\Unit m U$.
\end{itemizeminus}
We must prove the following properties: associativity and units for compositions, interchange and iterated units, compatibility of $\,\Cnx f, \Top, \Bot, \Triv$ with compositions and units, functoriality of $\,\cnx$ and naturality of $\Top, \Bot, \Triv$.

\begin{lemma} (functoriality) \label{lemma:functoriality}
Any $\omega$-functor $f : X \to Y$ extends to cylinders in a canonical way:
\begin{enumerateroman}
\item for any $n$-cylinder $U : x \cto x'$ in $X$, we get some $n$-cylinder $f \, U : f \, x \cto f \, x'$ in $Y$;
\item we have $f \, U \para f \, V$ whenever $U \para V$, and $f \, W : f \, U \to f \, V$ for any $W : U \to V$;
\item we have $(g \circ f) \, U = g \, f \, U$ for any $\omega$-functor $g : Y \to Z$, and also $\id \, U = U$.
\end{enumerateroman}
In other words, $\cnx$ defines a functor from $\ocat$ to $\globset$ and the homomorphisms $\Top, \Bot$ are natural.
\end{lemma}

\begin{defn} (left and right action)
Precomposition and postcomposition extend to cylinders. For any 0-cells $x, y, z$, we get:
\begin{itemizeminus}
\item the $n$-cylinder $u \act V$ in $\HOM x z$, defined for any 1-cell $u : x \to y$ and for any $n$-cylinder $V$ in $\HOM y z$;
\item the $n$-cylinder $U \act v$ in $\HOM x z$, defined for any 1-cell $v : y \to z$ and for any $n$-cylinder $U$ in $\HOM x y$.
\end{itemizeminus}
\end{defn}

\begin{lemma} (bimodularity) \label{lemma:bimodularity}
The following identities hold for any 0-cells $x, y, z, t$:
\begin{itemizeminus}
\item $(u \Comp 0 v) \act W = u \act (v \act W)$ for any 1-cells $u : x \to y$ and $v : y \to z$, and for any $n$-cylinder $W$ in $\HOM z t$;
\item $(U \act v) \act w = U \act (v \Comp 0 w)$ for any 1-cells $v : y \to z$ and $w : z \to t$, and for any $n$-cylinder $U$ in $\HOM x y$;
\item $(u \act V) \act w = u \act (V \act w)$ for any 1-cells $u : x \to y$ and $w : z \to t$, and for any $n$-cylinder $V$ in $\HOM y z$.
\end{itemizeminus}
Moreover, we have $\unit x \act U = U = U \act \unit y$ for any 0-cells $x, y$ and for any $n$-cylinder $U$ in $\HOM x y$.
\end{lemma}

This is proved by functoriality.

We omit parentheses in such expressions: For instance, $u \act v \act W$ stands for $u \act (v \act W)$, and $U \act v \act w$ for $(U \act v) \act w$. Moreover, action will always have precedence over other operations: For instance, $u \act V \comp W$ stands for $(u \act V) \comp W$.

\begin{defn} (concatenation)
By induction on $n$, we define the $n$-cylinder $U \comp V : x \cto z$ for any $n$-cylinders $U : x \cto y$ and $V : y \cto z$:
\begin{itemizeminus}
\item if $n = 0$, then $\Pal{(U \comp V)} = \Pal U \Comp 0 \Pal V$;
\item if $n > 0$, then $\Sce{(U \comp V)} = \Sce U \Comp 0 \Sce V$ and $\Tge{(U \comp V)} = \Tge U \Comp 0 \Tge V$, whereas $\Sht{U \comp V} = \Sht U \act \Tge V \comp \Sce U \act \Sht V$.
\end{itemizeminus}
In both cases, we say that $U$ and $V$ are \emph{consecutive}, and we write $U \cons V$.
\end{defn}

\begin{lemma} (source and target of a concatenation)
We have $U \comp U' \para V \comp V'$ for any $n$-cylinders $U \para V$ and $U' \para V'$ such that $U \cons U'$ and $V \cons V'$, and $W \comp W' : U \comp U' \to V \comp V'$ for any $n{+}1$-cylinders $W : U \to V$ and $W' : U' \to V'$ such that $W \cons W'$.
\end{lemma}

\begin{lemma} (compatibility of $\,\Cnx f$ with concatenation and $\Triv$)
The following identities hold any $\omega$-functor $f : X \to Y$:
\begin{itemizeminus}
\item $f (U \comp V) = f \, U \comp f \, V$ for any $n$-cylinders $U \cons V$ in $X$;
\item $f \, \Triv \, x = \Triv \, f \, x$ for any $n$-cell $x$ in $X$.
\end{itemizeminus}
In particular, the homomorphism $\Triv$ is natural.
\end{lemma}

In the cases of precomposition and postcomposition, we get the following result:

\begin{lemma} (distributivity over concatenation and $\Triv$) \label{lemma:distributivity}
The following identities hold for any 0-cells $x, y, z$ and for any 1-cell $u : x \to y$:
\begin{itemizeminus}
\item $u \act (V \comp W) = u \act V \comp u \act W$ for any $n$-cylinders $V \cons W$ in $\HOM y z$;
\item $u \act \Triv \Sht v = \Triv \Sht{u \Comp 0 v}$ for any $n{+}1$-cell $v : y \To 0 z$.
\end{itemizeminus}
There are similar properties for right action.
\end{lemma}

\begin{lemma} (associativity and units for concatenation)
The following identities hold for any $n$-cylinders $U \cons V \cons W$ and for any $n$-cylinder $U : x \cto y$:
\[
(U \comp V) \comp W = U \comp (V \comp W), \qquad
\Triv \, x \comp U = U = U \comp \Triv \, y.
\]
\end{lemma}
\begin{proof} We proceed by induction on $n$.

The case $n = 0$ is obvious.

If $n > 0$, the first identity is obtained as follows:
\begin{align*}
\Sht{(U \comp V) \comp W}
&= \Sht{U \comp V} \act \Tge W \comp \Sce{(U \comp V)} \act \Sht W
\tag{definition of $\comp$} \\
&= (\Sht U \act \Tge V \comp \Sce U \act \Sht V) \act \Tge W \comp (\Sce U \Comp 0 \Sce V) \act \Sht W
\tag{definition of $\comp$} \\
&= (\Sht U \act \Tge V \act \Tge W \comp \Sce U \act \Sht V \act \Tge W) \comp \Sce U \act \Sce V \act \Sht W
\tag{distributivity over $\comp$} \\
&= \Sht U \act \Tge V \act \Tge W \comp (\Sce U \act \Sht V \act \Tge W \comp \Sce U \act \Sce V \act \Sht W)
\tag{induction hypothesis} \\
&= \Sht U \act (\Tge V \Comp 0 \Tge W) \comp \Sce U \act (\Sht V \act \Tge W \comp \Sce V \act \Sht W)
\tag{distributivity over $\comp$} \\
&= \Sht U \act \Tge{(V \comp W)} \comp \Sce U \act \Sht{V \comp W}
\tag{definition of $\comp$} \\
&= \Sht{U \comp (V \comp W)}.
\tag{definition of $\comp$}
\end{align*}
The second identity is obtained as follows, using distributivity over $\Triv$ and the induction hypothesis:
\[
\Sht{\Triv \, x \comp U}
= \Sht{\Triv \, x} \act \Tge U \comp \Sce{(\Triv \, x)} \act \Sht U
= \Triv{\Sht x} \act \Tge U \comp \unit{\Sce x} \act \Sht U
= \Triv{\Sht{x \Comp 0 \Tge U}} \comp \Sht U
= \Sht U,
\]
and similarly for the third one.
\end{proof}

From now on, we shall omit parentheses in concatenations.

\begin{lemma} (cylinders in a cartesian product) \label{lemma:cartesian}
There are natural isomorphisms of globular sets $\Cnx{X \times Y} \simeq \Cnx X \times \Cnx Y$ and $\Cnx \SNG \simeq \SNG$, which satisfy the following coherence conditions with the canonical isomorphisms $(X \times Y) \times Z \simeq X \times (Y \times Z)$ and $\SNG \times X \simeq X \simeq X \times \SNG$:
\[
\begin{xy}
\xymatrix@=1em
{\Cnx{(X \times Y) \times Z} \ar[r] \ar[d] & \Cnx{X \times (Y \times Z)} \ar[d] \\
 \Cnx{X \times Y} \times \Cnx Z \ar[d] & \Cnx X \times \Cnx{Y \times Z} \ar[d] \\
 (\Cnx X \times \Cnx Y) \times \Cnx Z \ar[r] & \Cnx X \times (\Cnx Y \times \Cnx Z)}
\end{xy}
\qquad
\begin{xy}
\xymatrix@=1em
{\Cnx{\SNG \times X} \ar[r] \ar[d] & \Cnx X \ar@{=}[dd] & \Cnx{X \times \SNG} \ar[l] \ar[d] \\
 \Cnx \SNG \times \Cnx X \ar[d] & & \Cnx X \times \Cnx \SNG \ar[d] \\
 \SNG \times \Cnx X \ar[r] & \Cnx X & \Cnx X \times \SNG \ar[l]}
\end{xy}
\]
\end{lemma}

\begin{rem}
There is a coherence condition for the symmetry $X \times Y \simeq Y \times X$, but we shall not use it explicitly.
\end{rem}

\begin{rem}
By Lemmas~\ref{lemma:functoriality} and~\ref{lemma:cartesian}, any $\omega$-bifunctor $f : X \times Y \to Z$ extends to cylinders in a canonical way.
\end{rem}

\begin{defn} (multiplication)
Composition extends to cylinders: For any 0-cells $x, y, z$, we get the $n$-cylinder $U \COMP V$ in $\HOM x z$, defined for any $n$-cylinders $U$ in $\HOM x y$ and $V$ in $\HOM y z$.
\end{defn}

\begin{lemma} (associativity of multiplication)
The following identity holds for any 0-cells $x, y, z, t$, and for any $n$-cylinders $U$ in $\HOM x y$, $V$ in $\HOM y z$, $W$ in $\HOM z t$:
\[
(U \COMP V) \COMP W = U \COMP (V \COMP W)
\]
\end{lemma}
\begin{proof} By functoriality, using coherence with the canonical isomorphism $(X \times Y) \times Z \simeq X \times (Y \times Z)$.
\end{proof}

\begin{rem}
In $\Cnx{X \times Y} \simeq \Cnx X \times \Cnx Y$, concatenation and $\Triv$ can be defined componentwise.
\end{rem}

Using compatibility of $\,\Cnx f$ with concatenation and $\Triv$, we get the following result:

\begin{lemma} (compatibility of multiplication with concatenation and $\Triv$)
The following identities hold for any 0-cells $x, y, z$, for any $n$-cylinders $U \cons U'$ in $\HOM x y$ and $V \cons V'$ in $\HOM y z$, and for any $n{+}1$-cells $u : x \To 0 y$ and $v : y \To 0 z$:
\[
(U \comp U') \COMP (V \comp V') = (U \COMP V) \comp (U' \COMP V'), \qquad
\Triv \Sht u \COMP \Triv \Sht v = \Triv \Sht{u \Comp 0 v}.
\]
\end{lemma}

\begin{rem}
Any 0-cell $x$ in $X$ defines an $\omega$-functor $\sng x : \SNG \to X$, from which we get $\cnx \sng x : \SNG \simeq \Cnx \SNG \to \Cnx X$. It is easy to see that this homomorphism of globular sets corresponds to the sequence of trivial $n$-cylinders $\Triv \, \Unit n x$.
\end{rem}

\begin{lemma} (representability)
The following identities hold for any 0-cells $x, y, z$:
\begin{itemizeminus}
\item $u \act V = \Triv \, \Unit n {\Sht u} \COMP V = \Triv \Sht{\Unit {n{+}1} u} \COMP V$ for any 1-cell $u : x \to y$ and for any $n$-cylinder $V$ in $\HOM y z$;
\item $U \act v = U \COMP \Triv \, \Unit n {\Sht v} = U \COMP \Triv \Sht{\Unit {n{+}1} v}$ for any 1-cell $v : y \to z$ and for any $n$-cylinder $U$ in $\HOM x y$.
\end{itemizeminus}
In other words, the (left and right) action of a 1-cell $u$ is represented by the $n$-cylinder $\Triv \Sht{\Unit {n{+}1} u}$.
\end{lemma}
\begin{proof} By functoriality, using coherence with the canonical isomorphisms $\SNG \times X \simeq X \simeq X \times \SNG$. 
\end{proof}

\begin{defn} (extended action)
For any 0-cells $x, y, z$, we extend left and right action to higher dimensional cells as follows:
\begin{itemizeminus}
\item $u \act V = \Triv \Sht u \COMP V$ for any $n{+}1$-cell $u : x \to y$ and for any $n$-cylinder $V$ in $\HOM y z$;
\item $U \act v = U \COMP \Triv \Sht v$ for any $n{+}1$-cell $v : y \to z$ and for any $n$-cylinder $U$ in $\HOM x y$.
\end{itemizeminus}
\end{defn}

\begin{rem}
In particular, we get $u \act V = \Unit {n{+}1} u \act V$ for any 1-cell $u : x \to y$ and for any $n$-cylinder $V$ in $\HOM y z$, and
similarly for the right action. This means that we have indeed extended the action of 1-cells.
\end{rem}

\begin{lemma} (extended bimodularity)
The first three identities of lemma~\ref{lemma:bimodularity} extend to higher dimensional cells.
\end{lemma}
\begin{proof} By associativity of multiplication and compatibility of multiplication with $\Triv$.
\end{proof}

\begin{lemma} (extended distributivity)
The identities of lemma~\ref{lemma:distributivity} extend to higher dimensional cells.
\end{lemma}
\begin{proof} The first identity is obtained as follows, using compatibility of multiplication with concatenation:
\[
u \act (V \comp W)
= \Triv \Sht u \COMP (V \comp W)
= (\Triv \Sht u \comp \Triv \Sht u) \COMP (V \comp W)
= (\Triv \Sht u \COMP V) \comp (\Triv \Sht u \COMP W)
= u \act V \comp u \act W.
\]
The second one follows from compatibility of multiplication with $\Triv$.
\end{proof}

\begin{lemma} (commutation)
The following identities hold for any 0-cells $x, y, z$, for any $n{+}1$-cells $u, u' : x \To 0 y$ and $v, v' : y \To 0 z$, and for any $n$-cylinders $U : \Sht u \cto \Sht{u'}$ in $\HOM x y$ and $V : \Sht v \cto \Sht{v'}$ in $\HOM y z$:
\[
U \act v \comp u' \act V = U \COMP V = u \act V \comp U \act v'.
\]
\end{lemma}

\begin{proof} The first identity is obtained as follows, using compatibility of multiplication with concatenation:
\[
U \act v \comp u' \act V
= (U \COMP \Triv \Sht v) \comp (\Triv \Sht{u'} \COMP V)
= (U \comp \Triv \Sht{u'}) \COMP (\Triv \Sht v \comp V)
= U \COMP V,
\]
and similarly for the second one.
\end{proof}

From now on, we shall always assume that $m > n$.

\begin{defn} (compositions)
By induction on $n$, we define the $m$-cylinder $\Cto {U \Comp n V : R \To n T} {x \Comp n y} {x' \Comp n y'}$ for any $m$-cylinders $\Cto {U : R \To n S} x {x'} $ and $\Cto {V : S \To n T} y {y'}$:
\begin{itemizeminus}
\item $\Sce{(U \Comp 0 V)} = \Sce U = \Pal R$ and $\Tge{(U \Comp 0 V)} = \Tge V = \Pal T$, whereas $\Sht{U \Comp 0 V} = x \act \Sht V \comp \Sht U \act y'$;
\item if $n > 0$, then $\Sce{(U \Comp n V)} = \Sce U = \Sce V$ and $\Tge{(U \Comp n V)} = \Tge U = \Tge V$, whereas $\Sht{U \Comp n V} = \Sht U \Comp{n-1} \Sht V$.
\end{itemizeminus}
In both cases, we say that $U$ and $V$ are $n$-\emph{composable}, and we write $U \Cons n V$.
\end{defn}

\begin{lemma} (source and target of a composition)
We have $U \Comp n U' \para V \Comp n V'$ for any $m$-cylinders $U \para V$ and $U' \para V'$ such that $U \Cons n U'$ (so that $V \Cons n V'$), and $W \Comp n W' : U \Comp n U' \to V \Comp n V'$ for any $m{+}1$-cylinders $W : U \to V$ and $W' : U' \to V'$.
\end{lemma}

\begin{defn} (units)
By induction on $n$, we define the $m$-cylinder $\Cto {\Unit m U : U \To n U} {\Unit m x} {\Unit m y}$ for any $n$-cylinder $U : x \cto y$:
\begin{itemizeminus}
\item if $n = 0$, then $\Sce{(\Unit m U)} = \Tge{(\Unit m U)} = \Pal U$, whereas $\Sht{\Unit m U} = \Triv \Sht{\Unit m {\Pal U}}$. In particular, we get $\Sht{\Unit 1 U} = \Triv \Sht{\Pal U}$;
\item if $n > 0$, then $\Sce{(\Unit m U)} = \Sce U$ and $\Tge{(\Unit m U)} = \Tge U$, whereas $\Sht{\Unit m U} = \Unit {m{-}1} {\Sht U}$.
\end{itemizeminus}
\end{defn}

\begin{lemma} (source and target of a unit)
We have $\Unit {m{+}1} U : \Unit m U \to \Unit m U$ for any $n$-cylinder $U$.
\end{lemma}

\begin{rem}
By construction, $\Top$ and $\Bot$ are compatible with compositions and units.
\end{rem}

\begin{lemma} (associativity and units for compositions)
The following identities hold for any $m$-cylinders $U \Cons n V \Cons n W$ and for any $m$-cylinder $U : S \To n T$:
\[
(U \Comp n V) \Comp n W = U \Comp n (V \Comp n W), \qquad
\Unit m S \Comp n U = U = U \Comp n \Unit m T.
\]
\end{lemma}
\begin{proof} We proceed by induction on $n$.

If $n = 0$, the first identity is obtained as follows (with $U : x \cto x'$, $V : y \cto y'$ and $W : z \cto z'$):
\begin{align*}
\Sht{(U \Comp 0 V) \Comp 0 W}
&= (x \Comp 0 y) \act \Sht W \comp \Sht{U \Comp 0 V} \act z'
\tag{definition of $\Comp 0$} \\
&= x \act y \act \Sht W \comp (x \act \Sht V \comp \Sht U \act y') \act z'
\tag{definition of $\Comp 0$} \\
&= x \act y \act \Sht W \comp x \act \Sht V \act z' \comp \Sht U \act y' \act z'
\tag{distributivity over $\comp$} \\
&= x \act (y \act \Sht W \comp \Sht V \act z') \comp \Sht U \act y' \act z'
\tag{distributivity over $\comp$} \\
&= x \act \Sht{V \comp W} \comp \Sht U \act (y' \Comp 0 z')
\tag{definition of $\Comp 0$} \\
&= \Sht{U \Comp 0 (V \Comp 0 W)}.
\tag{definition of $\Comp 0$}
\end{align*}
The second identity is obtained as follows (with $U : x \cto y$ and $S : \Sce x \cto \Sce y$), using distributivity over $\Triv$:
\[
\Sht{\Unit m S \Comp 0 U} = \Unit m {\Sce x} \act \Sht U \comp \Sht{\Unit m S} \act y = \unit{\Sce x} \act \Sht U \comp \Triv \Sht{\Unit m {\Pal S}} \act y = \Sht U \comp \Triv \Sht{\Unit m {\Pal S} \Comp 0 y} = \Sht U,
\]
and similarly for the third one.

If $n > 0$, we apply the induction hypothesis.
\end{proof}

\begin{lemma} (compatibility of $\tau$ with compositions and units)
The following identities hold for any $m$-cells $u \Cons n v$ and for any $n$-cell $x$:
\[
\Triv(u \Comp n v) = \Triv \, u \Comp n \Triv \, v, \qquad
\Triv \, \Unit m x = \Unit m {\Triv \, x}.
\]
\end{lemma}
\begin{proof} By induction on $n$.

If $n = 0$, the first identity is obtained as follows, using distributivity over $\Triv$:
\[
\Sht{\Triv(u \Comp 0 v)}
= \Triv{\Sht{u \Comp 0 v}}
= \Triv{\Sht{u \Comp 0 v}} \comp \Triv{\Sht{u \Comp 0 v}}
= u \act \Triv \Sht v \comp \Triv \Sht u \act v
= u \act \Sht{\Triv \, v} \comp \Sht{\Triv \, u} \act v
= \Sht{\Triv \, u \Comp 0 \Triv \, v}.
\]
The second identity is obtained as follows:
\[
\Sht{\Triv \, \Unit m x}
= \Triv \Sht{\Unit m x}
= \Triv \Sht{\Unit m {\unit x}}
= \Triv \Sht{\Unit m {\Pal{(\Triv \, x)}}}
= \Sht{\Unit m {\Triv \, x}}.
\]
If $n > 0$, we apply the induction hypothesis.
\end{proof}

\begin{lemma} (compatibility of $\,\Cnx f$ with compositions and units)
The following identities hold any $\omega$-functor $f : X \to Y$:
\begin{itemizeminus}
\item $f (U \Comp n V) = f \, U \Comp n f \, V$ for any $m$-cylinders $U \Cons n V$ in $X$;
\item $f \, \Unit m U = \Unit m {f \, U}$ for any $n$-cylinder $U$ in $X$.
\end{itemizeminus}
\end{lemma}

In the cases of precomposition and postcomposition, we get the following result:

\begin{lemma} (distributivity over compositions and units)
The following identities for any 0-cells $x, y, z$ and for any 1-cell $u : x \to y$:
\begin{itemizeminus}
\item $u \act (V \Comp n W) = u \act V \Comp n u \act W$ for any $m$-cylinders $V \Cons n W$ in $\HOM y z$;
\item $u \act \Unit m V = \Unit m {u \act V}$ for any $n$-cylinder $V$ in $\HOM y z$.
\end{itemizeminus}
There are similar properties for right action.
\end{lemma}

\begin{lemma} (compatibility of concatenation with composition and units)
The following identities hold for any $m$-cylinders $U \Cons n V$ and $U' \Cons n V'$ such that $U \cons U'$ and $V \cons V'$, and for any $n$-cylinders $S \cons T$:
\[
(U \Comp n V) \comp (U' \Comp n V') = (U \comp U') \Comp n (V \comp V'), \qquad
\Unit m S \comp \Unit m T = \Unit m {S \comp T}.
\]
\end{lemma}
\begin{proof}
We proceed by induction on $n$.

If $n = 0$, the first identity is obtained as follows (with $U : x \cto x'$, $U' : x' \cto x''$, $V : y \cto y'$ and $V' : y' \cto y''$):
\begin{align*}
\Sht{(U \Comp 0 V) \comp (U' \Comp 0 V')}
&= \Sht{U \Comp 0 V} \act \Tge{(U' \Comp 0 V')} \comp \Sce{(U \Comp 0 V)} \act \Sht{U' \Comp 0 V'}
\tag{definition of $\comp$} \\
&= (x \act \Sht V \comp \Sht U \act y') \act \Tge{V'} \comp \Sce U \act (x' \act \Sht{V'} \comp \Sht{U'} \act y'')
\tag{definition of $\Comp 0$} \\
&= x \act \Sht V \act \Tge{V'} \comp \Sht U \act y' \act \Tge{V'} \comp \Sce U \act x' \act \Sht{V'} \comp \Sce U \act \Sht{U'} \act y''
\tag{distributivity over $\comp$} \\
&= x \act \Sht V \act \Tge{V'} \comp x \act \Sce V \act \Sht{V'} \comp \Sht U \act \Tge{U'} \act y'' \comp \Sce U \act \Sht{U'} \act y''
\tag{commutation} \\
&= x \act (\Sht V \act \Tge{V'} \comp \Sce V \act \Sht{V'}) \comp (\Sht U \act \Tge{U'} \comp \Sce U \act \Sht{U'}) \act y''
\tag{distributivity over $\comp$} \\
&= x \act \Sht{V \comp V'} \comp \Sht{U \comp U'} \act y''
\tag{definition of $\comp$} \\
&= \Sht{(U \comp U') \Comp 0 (V \comp V')}.
\tag{definition of $\Comp 0$}
\end{align*}
In the commutation step, we use the fact that $\Tge U = \Sce V$ and $\Tge{U'} = \Sce{V'}$ since $U \Cons 0 V$ and $U' \Cons 0 V'$.

The second identity is obtained as follows, using distributivity over $\Triv$:
\[
\begin{array}{c}
\Sht{\Unit m S \comp \Unit m T}
= \Sht{\Unit m S} \act \Tge{(\Unit m T)} \comp \Sce{(\Unit m S)} \act \Sht{\Unit m T}
= \Triv \Sht{\Unit m {\Pal S}} \act \Pal T \comp \Pal S \act \Triv \Sht{\Unit m {\Pal T}}
= \vspace{1ex} \\
\Triv \Sht{\Unit m {\Pal S \Comp 0 \Pal T}} \comp \Triv \Sht{\Unit m {\Pal S \Comp 0 \Pal T}}
= \Triv \Sht{\Unit m {\Pal S \Comp 0 \Pal T}}
= \Triv \Sht{\Unit m {\Pal{(S \comp T)}}}
= \Sht{\Unit m {S \comp T}}.
\end{array}
\]
If $n > 0$, the first identity is obtained as follows:
\begin{align*}
\Sht{(U \Comp n V) \comp (U' \Comp n V')}
&= \Sht{U \Comp n V} \act \Tge{(U' \Comp n V')} \comp \Sce{(U \Comp n V)} \act \Sht{U' \Comp n V'}
\tag{definition of $\comp$} \\
&= (\Sht U \Comp{n{-}1} \Sht V) \act \Tge{U'} \comp \Sce U \act (\Sht{U'} \Comp{n{-}1} \Sht{V'})
\tag{definition of $\Comp n$} \\
&= (\Sht U \act \Tge{U'} \Comp{n{-}1} \Sht V \act \Tge{U'}) \comp (\Sce U \act \Sht{U'} \Comp{n{-}1} \Sce U \act \Sht{V'})
\tag{distributivity over $\Comp{n{-}1}$} \\
&= (\Sht U \act \Tge{U'} \comp \Sce U \act \Sht{U'}) \Comp{n{-}1} (\Sht V \act \Tge{U'} \comp \Sce U \act \Sht{V'}) \tag{induction hypothesis} \\
&= \Sht{U \comp U'} \Comp{n{-}1} \Sht{V \comp V'}
\tag{definition of $\comp$} \\
&= \Sht{(U \comp U') \Comp n (V \comp V')}.
\tag{definition of $\Comp n$}
\end{align*}
In the penultimate step, we use the fact that $\Sce U = \Sce V$ and $\Tge{U'} = \Tge{V'}$ since $U \Cons n V$ and $U' \Cons n V'$.

The second identity is obtained as follows, using distributivity over units and the induction hypothesis:
\[
\begin{array}{c}
\Sht{\Unit m S \comp \Unit m T}
= \Sht{\Unit m S} \act \Tge{(\Unit m T)} \comp \Sce{(\Unit m S)} \act \Sht{\Unit m T}
= \Unit {m{-}1} {\Sht S} \act \Tge T \comp \Sce S \act \Unit {m{-}1} {\Sht T}
= \vspace{1ex} \\
\Unit {m{-}1} {\Sht S \act \Tge T} \comp  \Unit {m{-}1} {\Sce S \act \Sht T}
= \Unit {m{-}1} {\Sht S \act \Tge T \comp \Sce S \act \Sht T}
= \Unit {m{-}1} {\Sht{S \comp T}}
= \Sht{\Unit m {S \comp T}}.
\end{array}
\]
\end{proof}

\begin{rem}
In $\Cnx{X \times Y} \simeq \Cnx X \times \Cnx Y$, compositions and units can be defined componentwise.
\end{rem}

Using compatibility of $\,\Cnx f$ with compositions and units, we get the following result:

\begin{lemma} (compatibility of multiplication with compositions and units)
The following identities hold for any 0-cells $x, y, z$, for any $m$-cylinders $U \Cons n U'$ in $\HOM x y$ and $V \Cons n V'$ in $\HOM y z$, and for any $n$-cylinders $S$ in $\HOM x y$ and $T$ in $\HOM y z$:
\[
(U \Comp n U') \COMP (V \Comp n V') = (U \COMP V) \Comp n (U' \COMP V'), \qquad
\Unit m S \COMP \Unit m T = \Unit m {S \COMP T}.
\]
\end{lemma}

\begin{lemma} (compatibility of action with compositions and units)
The following identities hold for any 0-cells $x, y, z$, for any $m{+}1$-cells $u, u' : x \To 0 y$ such that $u \Cons{n{+}1} u'$, for any $m$-cylinders $V \Cons n V'$ in $\HOM y z$, for any $n{+}1$-cell $s : x \To 0 y$, and for any $n$-cylinder $T$ in $\HOM y z$:
\[
(u \Comp{n{+}1} u') \act (V \Comp n V') = u \act V \Comp n u' \act V', \qquad
\Unit {m{+}1} s \act \Unit m T = \Unit m {s \act T}.
\]
There are similar properties for right action.
\end{lemma}
\begin{proof} The first identity is obtained as follows, using compatibility of $\Triv$ with compositions and the previous lemma:
\[
\begin{array}{c}
(u \Comp{n{+}1} u') \act (V \Comp n V')
= \Triv \Sht{u \Comp{n{+}1} u'} \COMP (V \Comp n V')
= \Triv (\Sht u \Comp n \Sht{u'}) \COMP (V \Comp n V')
= \vspace{1ex} \\
(\Triv \Sht u \Comp n \Triv \Sht{u'}) \COMP (V \Comp n V')
= (\Triv \Sht u \COMP V) \Comp n (\Triv \Sht{u'} \COMP V')
= u \act V \Comp n u' \act V'.
\end{array}
\]
The second identity is obtained as follows, using compatibility of $\Triv$ with units and the previous lemma:
\[
\Unit {m{+}1} s \act \Unit m  T
= \Triv \Sht{\Unit {m{+}1} s} \COMP \Unit m  T
= \Triv \, \Unit m {\Sht s} \COMP \Unit m  T
= \Unit m {\Triv {\Sht s}} \COMP \Unit m  T
= \Unit m {\Triv {\Sht s} \COMP  T}
= \Unit m {s \act  T}.
\]
\end{proof}

Now we assume that $m > n > p$.

\begin{lemma} (interchange)
The following identities hold for any $m$-cylinders $U \Cons n U'$ and $V \Cons n V'$ such that $U \Cons p V$ (so that $U' \Cons p V'$), for any $n$-cylinders $S \Cons p T$, and for any $p$-cylinder $R$:
\[
(U \Comp n U') \Comp p (V \Comp n V') = (U \Comp p V) \Comp n (U' \Comp p V'), \qquad
\Unit m S \Comp p \Unit m T = \Unit m {S \Comp p T}, \qquad
\Unit m {\Unit n R} = \Unit m R.
\]
\end{lemma}
\begin{proof} We proceed by induction on $p$.

If $p = 0$, the first identity is obtained as follows (with $U : x \cto y$, $U' : x' \cto y'$, $V : z \cto t$ and $V' : z' \cto t'$):
\begin{align*}
\Sht{(U \Comp n U') \Comp 0 (V \Comp n V')}
&= (x \Comp n x') \act \Sht{V \Comp n V'} \comp \Sht{U \Comp n U'} \act (t \Comp n t')
\tag{definition of $\Comp 0$} \\
&= (x \Comp n x') \act (\Sht V \Comp{n{-}1} \Sht{V'}) \comp (\Sht U \Comp{n{-}1} \Sht{U'}) \act (t \Comp n t')
\tag{definition of $\Comp n$} \\
&= (x \act \Sht V \Comp{n{-}1} x' \act \Sht{V'}) \comp (\Sht U \act t\Comp{n{-}1} \Sht{U'} \act t')
\tag{compatibility of $\act$ with $\Comp{n{-}1}$} \\
&= (x \act \Sht V \comp \Sht U \act t) \Comp{n{-}1} (x' \act \Sht{V'} \comp \Sht{U'} \act t')
\tag{compatibility of $\comp$ with $\Comp{n{-}1}$} \\
&= \Sht{U \Comp 0 V} \Comp{n{-}1} \Sht{U' \Comp 0 V'}
\tag{definition of $\Comp 0$} \\
&= \Sht{(U \Comp 0 V) \Comp n (U' \Comp 0 V')}.
\tag{definition of $\Comp n$}
\end{align*}
The second identity is obtained as follows (with $S : x \cto x'$ and $T : y \cto y'$), using compatibility of action and concatenation with units:
\[
\begin{array}{c}
\Sht{\Unit m S \Comp 0 \Unit m T}
= \Unit m x \act \Sht{\Unit m T} \comp \Sht{\Unit m S} \act \Unit m{y'}
= \Unit m x \act \Unit {m{-}1} {\Sht T} \comp \Unit {m{-}1} {\Sht S} \act \Unit m {y'}
= \vspace{1ex} \\
\Unit {m{-}1} {x \act \Sht T} \comp \Unit {m{-}1} {\Sht S \act y'}
= \Unit {m{-}1} {x \act \Sht T \comp \Sht S \act y'}
= \Unit {m{-}1} {\Sht{S \Comp 0 T}}
= \Sht{\Unit m {S \Comp 0 T}}.
\end{array}
\]
The third identity is obtained as follows, using compatibility of $\Triv$ with units:
\[
\Sht{\Unit m {\Unit n R}}
= \Unit {m{-}1} {\Sht{\Unit n R}}
= \Unit {m{-}1} {\Triv \Sht{\Unit n {\Pal R}}}
= \Triv \, \Unit {m{-}1} {\Sht{\Unit n {\Pal R}}}
= \Triv \Sht{\Unit m {\Unit n {\Pal R}}}
= \Triv \Sht{\Unit m {\Pal R}}
= \Sht{\Unit m R}.
\]
If $p > 0$, we apply the induction hypothesis.
\end{proof}


\end{document}